\documentclass{jloganal}
\usepackage[all,2cell]{xy}
\usepackage{amsbsy}   

\usepackage{mathbbol}     
\usepackage{amsmath}

\usepackage{pinlabel}   

\DeclareMathAlphabet{\mathpzc}{OT1}{pzc}{m}{it}

\SetSymbolFont{letters}{bold}{OML}{cmm}{b}{it}   
\SetSymbolFont{operators}{bold}{OT1}{cmr}{bx}{n} 

\DeclareMathAlphabet{\mathitbf}{OML}{cmm}{b}{it} 

\usepackage[T1]{fontenc}
\usepackage{frcursive}

\usepackage{multicol}

 \usepackage{isomath}

\author{Tom McGaffey}
\givenname{McGaffey}
\surname{Tom}
\address{}
\email{}
\urladdr{}

  \title[Topologizing Rings of Map Germs]{Topologizing Rings of Map Germs:\\An Order Theoretic Analysis of Germs\\via Nonstandard Methods}
\keywords{map germ, topology, nonstandard}

\subject{primary}{msc2000}{57N17}

\subject{secondary}{msc2000}{26E35,06F30,30G06,54EXX,16W80}

\arxivreference{}
\arxivpassword{}

\volumenumber{}
\issuenumber{}
\publicationyear{}
\papernumber{}
\startpage{}
\endpage{}
\doi{}
\MR{}
\Zbl{}
\received{}
\revised{}
\accepted{}
\published{}
\publishedonline{}
\proposed{}
\seconded{}
\corresponding{}
\editor{}
\version{}

\def\a{\alpha}
\def\b{\beta}

\def\d{\delta}
\def\D{\Delta}
\def\e{\epsilon}
\def\f{\frac}                          
\def\g{\gamma}
\def\G{\Gamma}
\def\k{\kappa}

\def\l|{\left|}
\def\la{\lambda}
\def\La{\Lambda}

\def\Om{\Omega}
\def\om{\omega}

\def\ov{\overline}                   

\def\r|{\right|}
\def\s{\sigma}

\def\un{\underline}                  

\def\wh{\widehat}                    

\def\wt{\widetilde}                  
\def\x{\times}                       
\def\z{\zeta}
\def\({\left(}
\def\){\right)}
\def\[{\left[}
\def\]{\right]}
\def\<{\left<}
\def\>{\right>}

\newcommand{\bk}[1]{\langle #1\rangle}

\def\ra{\rightarrow}
\newcommand{\rz}{\raisebox{.2ex}{*}}

\def\SA{\mathcal A}
\def\SB{\mathcal B}
\def\SC{\mathcal C}

\def\SE{\mathcal E}

\def\SG{\mathcal G}
\def\SH{\mathcal H}
\def\SI{\mathcal I}
\def\SJ{\mathcal J}

\def\SL{\mathcal L}
\def\SM{\mathcal M}
\def\SN{\mathcal N}

\def\SP{\mathcal P}
\def\SQ{\mathcal Q}
\def\SR{\mathcal R}
\def\SS{\mathcal S}
\def\ST{\mathcal T}
\def\SU{\mathcal U}
\def\SV{\mathcal V}
\def\SW{\mathcal W}
\def\SX{\mathcal X}
\def\SY{\mathcal Y}

\def\Fd{\mathfrak d}

\def\Ff{\mathfrak f}

\def\Fj{\mathfrak j}

\def\Fn{\mathfrak n}

\def\Fr{\mathfrak r}
\def\Fs{\mathfrak s}
\def\Ft{\mathfrak t}
\def\Fu{\mathfrak u}
\def\Fv{\mathfrak v}

\def\FB{\mathfrak B}

\def\FD{\mathfrak D}

\def\FL{\mathfrak L}

\def\FN{\mathfrak N}

\def\FP{\mathfrak P}

\def\FU{\mathfrak U}

\def\FX{\mathfrak X}

\def\pzA{{\mathpzc{A}}}

\def\pzD{\mathpzc{D}}
\def\pzE{\mathpzc{E}}

\def\pzN{\mathpzc{N}}

\def\pzQ{\mathpzc{Q}}

\def\pzT{\mathpzc{T}}
\def\pzU{\mathpzc{U}}

\def\pzn{\mathpzc{n}}

\def\pzp{\mathpzc{p}}
\def\pzq{\mathpzc{q}}

\def\pzy{\mathpzc{y}}

\def\sbbd{{\mathbb d}}

\def\bbn{{\mathbb N}}

\def\bbq{{\mathbb Q}}
\def\bbr{{\mathbb R}}
\def\bbs{{\mathbb S}}

\def\bbx{{\mathbb X}}

\def\ssm{\smallsetminus}

\newcommand{\brkfl}[1]{\lfloor #1 \rfloor}

\newcommand{\norm}[1]{\left\|#1\right\|}

\def\k{\kappa}

\newtheorem{theorem}{Theorem}[section]
\newtheorem{lemma}{Lemma}[section]
\newtheorem{proposition}{Proposition}[section]
\newtheorem{corollary}{Corollary}[section]
\newtheorem*{notation}{Notation}
\theoremstyle{definition}
\newtheorem{definition}{Definition}[section]
\newtheorem*{remark}{Remark}
\newtheorem*{example}{Example}
\numberwithin{equation}{section}  

\begin{document}

\begin{abstract}
    Using nonstandard analysis we define a topology on the ring of
    germs of functions: $(\bbr^n,0)\ra(\bbr,0)$. We prove that this topology is absolutely convex, Hausdorff, that convergent nets of continuous germs have continuous germs as limits and that, for continuous germs, ring operations and compositions are continuous. This topology is not first countable, and, in fact, we prove that no good first countable topology exists.  We give a spectrum of standard working descriptions for this topology. Finally, we identify this topological ring as a generalized metric space and examine some consequences.
\end{abstract}

\maketitle

\section{Introduction}\label{sec: intro: top on map germs}
    Two real-valued functions defined near $0$ in $\bbr^n$ are called {\it germ equivalent} if they coincide on some neighborhood of $0$. The equivalences classes are called {\it germs}.
    It is commonly believed that one cannot construct a nondiscrete ``good'' topology on the ring of germs at $0$ of smooth real valued functions on $\bbr^n$, much less the ring of germs of continuous real valued functions on $\bbr^n$.
    For example, Gromov, \cite{Gromov1986} remarks (p.36) that ``There is no useful topology in this space...of germs of [$C^k$] sections...'' over a particular set.
    Furthermore, there are hints in the literature, for example in the work of eg., Du Pleisses and Wall, \cite{DuPlessisWall1995},  on topological stability, see p.95 and chapter 5 (p.121-) on the great difficulties of working with germ representatives with respect to aspects of smooth topology (eg., how to define the stability of germs), but that there are no alternatives, eg., to working with the germs directly.

    Although, there are good topologies for vector spaces of very special map germs in terms of inductive limits of Frechet spaces, the constructions of these topologies depend on the nature of analytic germs; see eg., Meise and Vogt, \cite{MeiseVogt97}, and the discussion in the conclusion of this paper.
    In this paper, using nonstandard analysis, we will give a construction of an absolutely convex, Hausdorff topology on the ring of germs of real valued functions on $\mathbb R^n$ at $0$ that has good algebraic, compositional and convergence properties, and prove that, categorically, it is the best possible.

    More specifically, we give a construction of a non-first countable, Hausdorf topology on the ring of real valued germs on $\bbr^n$ at $0$ that has the following properties.  Analogous to norm convergence on a compact set, a convergent net of germs of continuous functions has limit the germ of a continuous function. Moreover, germ product and composition are continuous, with composition on the right by germs of homeomorphisms giving topological ring isomorphisms and germs of self homeomorphisms of $(\bbr^n,0)$ forming a topological  group. We also show that germ product with monotone germs has good Krull type topological behavior.
    Furthermore, we prove that this is the best possible result; that is, there is no first countable topology on the ring of germs making it into a Hausdorff topological ring if it is even weakly compatible with the intrinsic partial order on function germs. Our topology is strongly compatible with this order. (We thank Mike Wolf for asking this question.)
    Finally, although this work fundamentally depends on nonstandard tools, we finish with a multiplicity of standard descriptions for our topology $\tau$, eg., see  \autoref{cor: all standard descriptions tau} and the definitions and constructions leading up to it.
    In the conclusion, we show how our topology is a kind of generalized metric space that is homeomorphic to a generalized ultrametric space and examine some consequences.
    We also indicate how the constructions here might have utility for the construction of good topologies on germs in a variety of settings.


    Our constructions rely critically on nonstandard methods. Given the nonstandard work on germs stretching from the 1960's (see  Robinson, \cite{Robinson1969}) up to the present day (see eg., Kalfallah and Kosarew, \cite{KhalKos2010}), it's rather curious that  nonstandard investigations of germs of general functions (eg., beyond those with analytic type rigidities) has not occurred. Further, note that those constructions of nontrivial topologies on germs that do exist, ie., on analytic germs, do not involve nonstandard techniques.
    Nevertheless, it is within this arena of objects with vanishing domains that the nonstandard intuition (and it's rigorous machinery!)  become instrumental to well defining a nontrivial topological structure.
    In fact, the simple initial idea here was that germs at $0$, and distances between them, are well defined and determined on an infinitesimal ball. Bringing this simple intuition together with the machinery of nonstandard methods, allowed the author to get a toe hold on these elusive objects.
    Using this and the focusing consultations with Robert Hardt, the author was able to develop a firm theoretical framework for a useful standard notion of nearness for germs  of functions $f:(\bbr^n,0)\ra(\bbr^p,0)$ that is easily generalized to more general topological spaces (see our conclusion). Note that the family of germs of maps $f:(\bbr^n,0)\ra\bbr^p$, ie., germs  with variable target, is a remarkably different object and will not be considered in this paper.  See the remarks in the conclusion for this and questions of topologies on differentiable germs, part of the author's original motivation.
    The standard descriptions of convergence in this topology (the most transparent being the last description in \autoref{cor: all standard descriptions tau}), seem, after some thought, to be quite natural; nevertheless the author believes that standard proofs of the results of this paper will probably be difficult; see the remarks in the conclusion.
    Ironically, although nonstandard methods were central to the discovery of the topology $\tau$ and its properties, in this paper we do not do the usual thing and describe the (nonstandard) monads of $\tau$. Simply put, the technical development for the (nonstandard) description of the topological monads was not necessary for the construction of our topology $\tau$ or to verify its properties. On the other hand, it would not be surprising to find that given good descriptions of these monads, some of the proofs here could be distinctively shortened.

     Let us next list the technical obstacles and desired results to achieve our goals. Our topologies are simply defined as nonstandard sup norm topologies on an arbitrary, but fixed, infinitesimal ball. Simple nonstandard arguments show that these are well defined for germs. But unlike the standard case, it's  a serious problem \textbf{(1)}: to determine a choice of a family of (infinitesimal) bounds  $\G=\{\g\}$ such that if $\|\rz f\|_\d$ (the nonstandard supremum norm on a fixed  ball of radius $\d\sim 0$) is controlled by these specific values, $\G$, then  we get a good topological notion of nearness. We need the nearness structure defined in (1), to \textbf{(2)}:  be compatible with the requirement that it has good convergence properties, eg., we want a convergent net of continuous germs to have a continuous germ as the limit point.  As this topology is defined via a `norm' on a ball with radius some positive infinitesimal $\d$,  we want \textbf{(3)}: this topology to be independent of the choice of this infinitesimal. To be a good topology,  we need not only that it have good convergence properties, but also  that \textbf{(4)}: this topology has good properties with respect to ring operations and composition of germs. It turns out that our topology is not first countable, eg., one needs uncountable nets to get  convergence, so we want \textbf{(5)}: to show that any good topology on germs cannot be first countable. Finally, in order to be applicable, \textbf{(6)}: we need a good standard description of this topology.

     To have some chance for success in confronting the above problems, we needed the following. First, the algebra of germs at $0$ is canonically order isomorphic (via the domain restriction map) to the external algebra of standard functions on any infinitesimal ball about $0$, see  \autoref{cor: SG_0 -> F(B_delts) is R alg isomorph}.  Second, we define the germ topology  in terms of these nonstandard algebras of standard functions restricted to one of these infinitesimal balls, see \autoref{def: of topology  at 0 germ}. Third, we have a criterion, in the context of these  functions restricted to these infinitesimal balls,  to determine those germs whose representatives are continuous functions (on some neighborhood of $0$), see \autoref{prop: germ continuity from k<<<d}. Fourth, we employ a remarkable trick (possibly due to Hirschfeld, \cite{Hirschfeld1988}), used to full force in \autoref{prop: ptwise bndd subset unif bndd}, along with a trick adapted from Hardy, \cite{Hardy1924}, to verify that the topology is independent of the radius of the given infinitesimal ball to which  we restrict our standard functions. Fifth, the non-first countability of a good topology follows from an argument that depends on the existence of uncountably long almost chains of germs, see \autoref{subsec: ord preserv tops not countable}. Finally, the multiplicity of (standard) descriptions of our topology depends on the facts that our topology $\tau$ is independent of the kind of infinitesimal $\d$ used to describe it and also on the variety of moduli used to measure nearness of germs, see \autoref{sec: standard interpret tau}.

    Let's describe in some detail how we solve all of these problems. As an overview, we should note that throughout this paper, our problems are reduced to problems on the semiring, $\SM$, of positive monotone increasing germs defined on the positive real numbers, $\bbr_+$,  near $0$ (having limit $0$ at $0$). Essentially, the values at $\d$ of sufficiently numerous subsets of $\SM$ play the role of good sets of moduli. The monotonicity of elements of $\SM$ and their roles as models for the (transfers of the) functions $t\mapsto\|f\|_t$, make them our central focus; hence our concern with ordering of germs (see below).
     Without initially thinking of problem (2) we find that to get a good set of distances for $\rz\|\;\|_\d$,  we must, in fact,  define multiple families  of infinitesimals,  the $\SN_\d$ families, ($\d$ a positive infinitesimal), see eg.,  \autoref{def: SN_delta families of numbers}, (and the ${}_\d\SP$ later (see \autoref{def: PL^0 germs at 0}) and show that
     they define ``equivalent'' sets of distances (this is defined in terms of various  ``coinitial'' subsets of certain (totally ordered) subsets of the nonstandard real numbers), analogous to the trivial standard fact that eg., $1/2,1/3,1/4,\ldots$ and $1/4,1/9,1/16,\ldots$ form equivalent sets of distances for, say, the $\sup$ norm for the continuous functions on the unit ball (see  \autoref{def: of coinitial partial order relations}). (Note that, unlike this  simple example, none of our good sets of distances can be countable, which will imply that our topology is not first countable;  see \autoref{prop: SPSL has no countable coin subset}.)
    But the moduli chosen in problem (1) must satisfy problem (2) and this depends on an understanding of (2) that is compatible with the setting in (1). We are able to simplify the relation between (1) and (2) by reducing the conditions for a germ to be continuous (ie., have a neighborhood of $0$ and a  representative that is continuous there) to a tight bounding of differences of germ values on the given $\d$-ball. This is \autoref{prop: germ continuity from k<<<d}. We then needed to further transform this new criterion for germ continuity into one that involved (in a coinitial way) the moduli for our generalized norm $\rz\|f\|_\d$. This occurs in \autoref{cor: [f] good coin subsets->f is C^0}. From this, we are able to verify that net convergence in $\rz\|\;\|_\d$ with our choice of moduli, $\SN_\d$, is good enough to enforce continuity in the limit, see \autoref{prop: [fd]->[g]--> [g] is cont} and it's easy but fundamental consequence (once we have defined our topology $\tau^\d$),  \autoref{thm: convergnet of cont germs in SG converges to cont}.


    Originally, in \cite{McGaffeyPhD}, our solution to problem (3) followed  from a series of constructions passing from one family of interlocking monotone germs to another finally reducing the problem to a particular family of monotone analytic germs constructed by Hardy in \cite{Hardy1924} where it became solvabe in terms of monotone sequences of exponents. (Some of this material will appear elsewhere.) Eventually, we realized that using a generalized version of the monotonicity trick of Hardy in tandem with a powerful nonstandard device of Hirschfeld in \cite{Hirschfeld1988}, we were able to prove a general, all purpose pointwise pinching implies uniform pinching result for families of function germs, \autoref{prop: ptwise bndd subset unif bndd}. Using this result and some related results, eg., \autoref{prop: un(SM)^r subset un(SM)(m)},  we quickly get as a consequence, the independence of our topology of the choice of the infinitesimal $\d$, see \autoref{thm: tau^d on SM indep of d} and \autoref{cor: tau^d_1 = tau^d_2}.
    The solution to problem (5) occurs in the following part,  \autoref{subsec: ord preserv tops not countable}, and involves a further development of material used to solve  problem (3). Our earlier  work, \cite{McGaffeyPhD}, around  \autoref{prop: ptwise bndd subset unif bndd} allows the realization that there are (lots of) families of germs, eg., our piecewise affine germs $\SP\SL^0$, see \autoref{def: PL^0 germs at 0}, that are asymptotically order isomorphic to certain good coinitial subsets of our (totally ordered) set of moduli, $\SN_\d$, see \autoref{def: asymp tot ord set of germs} for this notion. Furthermore, the work around proposition \autoref{prop: ptwise bndd subset unif bndd} implies that our topology $\tau$, strongly respects this germ order, see \autoref{prop: tau preserves germ scales}. We say that the topology ``preserves germ scales'' (\autoref{def: pzT preserves germ scales}). It is then easy to  show that any reasonable topology that preserves germ scales cannot be first countable, the ordered tower of germs is too deep (\autoref{thm: pzT has germ scales->not 1st countable}.

    Problems (4) and (6) occupy the last two working sections of this paper. For problem (4), the topological aspects of the ring operations  for the space of continuous map germs occupy \autoref{subsec: top properties ring structure}; eg see  \autoref{prop: SG_0 is Hausdorff top ring} and the preceding lemma. The material on composition of continuous germs occurs in \autoref{subsec: top properties of germ composition}; see \autoref{prop: rc_h:G^0_n->G^0_n is C^0} for the continuity of right composition and for the statement of continuity of left composition, see \autoref{prop: lc_h:G^0_n,p,0->G^0_n,p,0 is C^0}, which follows from our second criterion for germ continuity, \autoref{prop: 2nd criterion for germ in C^0}. At the end of this section, we extend these results to germs of homeomorphisms, proving that the group of germs of homeomorphisms of $(\bbr^n,0)$ is a topological group, \autoref{prop: SH^0_n is top group}. To indicate the sensitivity our our topology with respect to continuity, we will also show that composition with a germ that is continuous at $0$, but is a noncontinuous germ does not give a continuous mapping between our rings of germs, see the
    \hyperlink{example}{Example}.
    Note that also in this section we  prove that multiplication by good monotone germs is an open map, giving our topology  a Krull topology flavor, see  \autoref{prop: mult SG by elt SM is open}.
    We believe that the solution to problem (6) gives satisfactory standard characterizations for our topology $\tau$. All are phrased in terms of standard conditions for the convergence of a net of germs. The variety of equivalent (nonstandard) moduli for our topology as well as the independence of the choice of infinitesimal $\d$ allowed both (apparently) weak and strong descriptions, see  \autoref{cor: all standard descriptions tau} for five descriptions varying from weak to strong. One who reads this paper can use this same freedom of choice of moduli and $\d$'s to construct (apparently many) other characterizations.

    We say order theoretic for an obvious reason: the intrinsic ordering of the sets of moduli, eg., $\SN_\d$ determines the topology. But also order plays a more subtle role in `depth' of neighborhoods of germs. It turns out that the total ordering of $\SN_\d$ (without coinitial countable subsets) induces asymptotically totally ordered (almost) chains of germs (in terms of their intrinsic partial order) with the property that these (almost) chains don't have countable coinitial subsets. The existence of these almost chains of germs implies that our topology or any other (that satisfies mild assumptions with respect to germ order) cannot be first countable. (For information on the ordering of $\rz\bbr$, ``the'' nonstandard real line, see the paper of Di Nasso and Forti, \cite{DiNassoForti2002}.) Note, therefore, that the Hahn field expression for the field completion of our semiring $\SN_\d$ has the property that its group of magnitudes does not have a countable coinitial subset, unlike some Hahn fields of recent interest, eg., see eg.,
    Todorov and Wolf, \cite{oai:arXiv.org:math/0601722},
    or Pestov, \cite{MR1135229}.
    The uncountable nature of coinitial subsets of $\SN_\d$ implies that our topology is what is termed an $\om_1$-metric space, see eg., Artico and Moresco, \cite{ArticoMoresco1981}. In the \hyperlink{topological}{Conclusion}, we see that this implies eg., that our space of germs is homeomorphic to a generalized ultrametric space, eg., is $0$ dimensional.

\section{Germs, infinitesimal moduli, continuity}\label{sec: various radii, converg of C^0 is C^0}
    In this section, we first develop algebraically and topologically faithful representations of the ring of function germs at $0$ in $\bbr^n$. In the next subsection, we begin an investigation of possible sets of nonstandard moduli proving that there are sets of moduli that give a good measure for convergence of germs, moduli that are good for determining the continuity of germs and that these different sets of moduli are sufficiently compatible to allow the definition of a topology, $\tau^\d$, incorporating both. This definition occurs at the beginning of the next section.
\subsection{Germs and their faithful nonstandard restrictions}
   We will be using nonstandard mathematics throughout this paper. For an impressionistic introduction, the reader might read \autoref{appendix: nsa impression}. Here we will give some notation. If $X$ is a set, $\rz X$ will denote it's transfer and ${}^\s X$ will denote the external copy of $X$ in $\rz X$, ie., ${}^\s X=\{\rz x:x\in X\}$.

   We will begin with a useful description of germs and give an easy result indicating that they are faithfully determined algebraically on infinitesimal balls.
\begin{notation}\label{notation: mu(0), B_d,etc}
    $\bbr$  denotes the real numbers with $\bbr_+=\{t\in\bbr:t>0\}$ and $\bbn$ the positive integers. $\rz\bbr$ denotes `the' nonstandard real numbers in a sufficiently saturated model of analysis, with $\rz\bbr_+=\{\Ft\in\rz\bbr:\Ft>0\}$. $\rz\bbn$ will denote the nonstandard positive integers with $\rz\bbn_\infty$ the infinite integers. Let $n\in\bbn$ and if $0<r\in\bbr$, respectively $0<\Fr\in\rz\bbr$, let $B_r=B^n_r=\{x\in\bbr^n:|x|\leq r\}$, respectively $\rz B^n_\Fr=\{\xi\in\rz\bbr^n:|\xi|\leq\Fr\}$. Let $\mu(0)=\mu_n(0)=\{\xi\in\rz\bbr^n:|\xi|\sim 0\}$, ie., infinitesimal vectors in $\rz\bbr^n$, and $\mu(0)_+=\{\xi\in\mu(0):\xi>0\}$, the positive infinitesimal nonstandard real numbers; we will sometimes write $0<\d\sim 0$ instead of $\d\in\mu(0)_+$.
\end{notation}
\begin{definition}\label{def: set of (f,U) s.t. U=domf}
    Let
\begin{align}
    \un{F}=\un{F}(n,1)=\{(U,f):U\subset\bbr^n\;\text{is a convex nbd of}\;\;0\;\text{and}\;\;f:U\ra\bbr\}
\end{align}
     and $\un{F}(n,1)_0\subset\un{F}(n,1)$ denote the set of those $(U,f)$ such that $f(0)=0$.
     For the associated set of germs of equivalence classes, let $\SG=\SG_{n,1}$ denote the ring of germs of $f:(\bbr^n,0)\ra(\bbr,0)$ at $0\in\bbr^n$, that $\SG^0\subset\SG$ the subring consisting  of germs of continuous functions.
     If $f=(U,f)\in\un{F}$, then we denote the germ equivalence class it belongs to by $[f]$. In certain circumstances, we will sometimes use $f$ for both the element of $\un{F}$ and for the germ class in $\SG$ it belongs to.
\end{definition}
    There is an important partial order on germs given as follows.
\begin{definition}\label{def: [f]<[g]}
    For $[f],[g]\in\SG$, we say that $[f]$ is less that  $[g]$ and write $[f]<[g]$ if there is a  neighborhood $U$ of $0$ in $\bbr^n$ and representatives $f\in[f]$, $g\in[g]$ defined on $U$ and satisfying $f(x)<g(x)$ for all $x\in U$. Define $[f]\leq[g]$, $[f]>[g]$ similarly.
\end{definition}
    Of course, by suitable restriction, we may assume that $U$ is appropriately nice, eg., convex.
    It is easy to see that this is well defined and indeed a partial order, ie., satisfies reflexivity, antisymmetry and transitivity.

    Although elementary, the following basic result is apparently folklore. There are many variations of this; the statement below is needed to prove that a germ at $0$, and its order class, is unambiguously defined by its restriction to an infinitesimal neighborhood of $0$.

\begin{lemma}\label{lem: *A contains inf nbd-> has stand nbd}
    Suppose that $A\subset\bbr^n$ and $0<\d\sim 0$ are such that $\{\xi\in\mu(0):|\xi|>\d\}\subset\rz A$. Then there is $0<r\in\bbr$ such that $B_r\ssm\{0\}\subset A$. Similarly, if $B\subset\bbr^n$ is such that $\rz B_\d\subset\rz B$, then there is $0<r\in\bbr$ such that $B_r\subset B$.
\end{lemma}
\begin{proof}
    First, it's clear that as $\rz A$ is internal, then overflow implies that there is a standard $a>0$ such that $\{\xi:\d<|\xi|\leq \rz a\}\subset \rz A$. Let $E$ denote $A\cap B_a$ and let $\complement E$ denote the complement of $E$ in $B_a\ssm\{0\}$; so that $E\cup\complement E=B_a\ssm\{0\}$.  We know that $\rz\complement E\subset B_\d\ssm\{0\}$; that is, for $0<d\in\bbr$, we have the statement: $\xi\in\rz\complement E\Rightarrow \xi\in\rz B_d\ssm\{0\}$. But then reverse transfer gives the statement: $x\in\complement E\Rightarrow x\in B_d\ssm\{0\}$ and as $d>0$ in $\bbr$ was arbitrary, then we get that $\complement E=\emptyset$ so that $E=B_d\ssm\{0\}$.
    To prove the second assertion, suppose the conclusion does not hold so that there is a maximum positive $\d\sim 0$ such that if $\SB_t$ is the set $\{t\in\bbr_+:B_t\subset B\}$, then
\begin{align}
    [0,\d)\subset\rz\SB\dot=\{\Ft\in\rz\bbr_+:\rz B_\Ft\subset\rz B\}.\notag
\end{align}
     But then $\{\Ft:\d<\Ft\sim 0\}\subset\rz\complement\SB$ and so as $\d\sim 0$ and is nonzero, the first part ($n=1$ here) implies that $\{\Ft:0<\Ft\sim 0\}\subset\rz\complement \SB$, forcing $[0,\d)\nsubseteq \rz \SB$, ie., $\rz B_\Ft\subsetneqq \rz B$ for $\Ft<\d$, a contradiction.
\end{proof}
    We will now introduce the far more workable nonstandard counterparts to the above germ definitions.
\begin{definition}
     Let $\rz F(B_\d)$ denote the $\rz\bbr$ algebra of internal functions on $B_\d$ and $^\s F(B_\d)$ denote the (external) subring of standard functions on $B_\d$.
     Given $\rz f|_{B_\d}$ and $\rz g|_{B_\d}$ in ${}^\s F(B_\d)$, we say that $f$ is less than $g$ on $B_\d$, written $\rz f<\rz g$ on $B_\d$ or $f<g$ on $B_\d$, if for every $\xi\in B_\d$, $\rz f(\xi)<\rz g(\xi)$.
\end{definition}
    It's clear that this defines a partial order on ${}^\s F(B_\d)$.

     Note that $\SG$ and its subring $\SG^0$ are $\bbr$ algebras. Furthermore, ${}^\s F(B_\d)$ is clearly an $^\s\bbr$ algebra , and so can be viewed as an $\bbr$ algebra.
     Given this, the above lemma has the following immediate consequence which is the critical fact that allows the characterizations of germs in this paper.
\begin{corollary}\label{cor: SG_0 -> F(B_delts) is R alg isomorph}
    Suppose that $U\subset\bbr^n$ is a neighborhood of $0$ in $\bbr^n$, $f: U\ra\bbr$ is a function and $0<\d\sim 0$. If $\rz f(\xi)=0$ for all $\xi\in B_\d$, then there is another neighborhood of $0$, $V\subset\bbr^n$ such that $f|_V$ is identically zero; ie., $[f]\in\SG$ is the zero germ.
    That is, the map $\SR_\d:\SG\ra\;^\s F(B_\d):[f]\mapsto \rz f|_{B_\d}$ is an $\bbr$-algebra isomorphism.
    Furthermore, this map is also a partial order isomorphism; ie., we have that $[f]<[g]$ if and only if $\rz f<\rz g$ on $B_\d$, with identical statements for $\leq,\;>$ and $\geq$.
\end{corollary}
\begin{proof}
    Let $supp(f)\subset U$ be the set of $x\in U$ such that $f(x)\not=0$ and  $A\subset U$ denote $U\ssm supp(f)$. Then $B_\d\subset\rz A$ and so the above lemma implies that there is a positive $r\in\bbr$ such that $B_r\subset A$, eg., $f(x)=0$ for $x\in B_r$; eg., $[g]=0$.
    To verify that $\SR_\d$ is an $\bbr$-algebra homomorphism is straightforward as $[f],[g]\in\SG$ and $c\in\bbr$ satisfy $c[f]=[cf],[f]+[g]=[f+g]$ and $[f][g]=[fg]$ and eg., $(\rz f\cdot\rz g)|_{B_\d}=(\rz f|_{B_\d})(\rz g|_{B_\d})$ as internal functions on $B_\d$.
    Clearly, $[f]<[g]$ gets $\rz f<\rz g$ on $B_\d$ by transfer and restriction; but the above lemma gets the reverse implication as $\rz\{x\in\bbr^n: f(x)< g(x)\}\supset B_\d$.
\end{proof}
   \textbf{ Given the above proposition, when we talk about germs or elements of $\pmb{\SG}$, we will usually be working with subalgebras of the external algebras $\pmb{^\s F(B_\d)}$.  In particular, although we will use the germ notations, $\SG$, etc.,  all work on germs will implicitly occur in the algebras $\pmb{^\s F(B_\d)}$, for some fixed positive infinitesimal $\pmb{\d}$.}

\subsection{Monadic determination of germ regularity}\label{subsec: monadic regularity of standard fcns}
     Here, we will give an infinitesimal criterion for germ continuity. It says that if a standard function defined on some neighborhood of $0$ satisfies a particular set of conditions on a given $B_\d$, $\d$ arbitrary, then, in fact, that function is continuous on some neighborhood of $0$.

    For perspective, we begin with the following two simple but surprising statements.
\begin{proposition}
    Suppose that $0<\d\sim 0$ and $[f]\in\SG$ is such that $\rz f|_{B_\d}$ is *continuous on $B_\d$. Then $[f]\in\SG^0$.
\end{proposition}
\begin{proof}
     The proof is trivial: if $A=\{r\in\bbr_+:f|_{B_r}\;\text{is continuous on}\; B_r\}$, then $\rz A=\{\Fr\in\rz\bbr_+:\rz f|_{\rz B_\Fr}\;\text{is *continuous on}\rz B_\Fr\;\}$ and the hypothesis says that $\rz A\not=\emptyset$ and so $A\not=\emptyset$.

\end{proof}

    The following corollary indicates that the topology we define on germs will be independent of the infinitesimal neighborhood.
\begin{corollary}\label{cor: *cont on B epsilon iff *cont on B del}
    Suppose that $[f]\in\SG$ and $\e,\d$ are positive infinitesimals. Then $\rz f$ is *continuous on $B_\d$ if and only if it is *continuous on $B_\e$.
\end{corollary}
\begin{proof}
    This is clear from the previous proposition.
\end{proof}
\begin{remark}
     Analogues of these two results for various regularity notions, eg., for homeomorphism germs, or differentiability classes, eg., germs of $C^k$ submersions, hold by almost identical arguments. We will return to these and their implications in later sections and in following papers.
     These results will allow one to work on monads where domains and ranges for standard functions are remarkably well defined and then lift to local standard results.
\end{remark}

       We will now define a stringent condition (over some infinitesimal $\rz B_\d$) on a germ $[f]$  with respect incomparable pairs (see below) of positive infinitesimals that forces $f$ to be continuous on a standard neighborhood. We begin with definitions.
\begin{definition}\label{def: e is [f]-good}
     Given $0<\d\sim 0$, we say that $\boldsymbol{\k\in\mu(0)_+}$ \textbf{is incomparably smaller} \textbf{than} $\boldsymbol{\d}$ if for all $m\in\SM$, we have $\rz m(\d)>\k$ (see \autoref{def:  SM and SM^0} for $\SM$). We may write this $\boldsymbol{\k\lll\d}$. Similarly, if $\om,\Om\in\rz\bbn$ are infinite integers, ie., elements of $\rz\bbn_\infty$, then we say that $\boldsymbol{\Om}$ \textbf{is} \textbf{incomparably larger than} $\boldsymbol{\om}$, written $\boldsymbol{\om}\mathbf{\lll}\boldsymbol{\Om}$, if for every monotonically increasing  $f:\bbn\ra\bbn$, we have $\rz f(\om)<\Om$.
    Let $\k,\d$ be positive infinitesimals, ie., $\k,\d\in\mu(0)_+$ with $\k\lll\d$. For a nonzero germ $[f]\in\SG$, we say that $\boldsymbol{\k}$ \textbf{is strongly} $\mathbf{[f]}$\textbf{-good for} $\boldsymbol{\d}$ if  the following holds.  For all \;$\xi,\z\in\rz B_\d$ with $|\xi-\z|$ sufficiently small, $|\rz f(\xi)-\rz f(\z)|<\k$ holds.
    We say that $\mathbf{[f]\in}\boldsymbol{\SG}$ \textbf{has strong} $\boldsymbol{\d}$\textbf{-good numbers} if  for some $\k\lll\d$, $\k$ is strongly $[f]$-good for $\d$.
\end{definition}
     Note that  for all $m,n\in\bbn$, $\om\lll\Om$ if and only if $m\om\lll n\Om$.
     A proof of the existence of incomparable pairs of positive infinitesimals is an easy concurrence argument in an enlarged model, see p 91 in Hurd and Loeb, \cite{HurdLoeb1985}.
     The following is an unsurprising, but useful correspondence between the relations $\om\lll\Om$ and $\k\lll\d$.
\begin{lemma}\label{lem:d>>>k -> [1/d]<<<[1/k]}
     If $t\in\bbr_+$, let $\brkfl{t}$ denotes the least integer $\geq t$.
     Then   $\k\lll\d$ (in $\mu(0)_+$)  implies that $\brkfl{1/\k}\ggg\brkfl{1/\d}$ as elements of $\rz\bbn_\infty$.
\end{lemma}
\begin{proof}
     Let
     $\om=\brkfl{1/\d}$ and $\Om=\brkfl{1/\k}$ and note that $\om\lll\Om$ if and only if $\om/2\lll 2\Om$. (Note here that we may assume that $\om$ is an even integer.) So, suppose contrary to conclusion that there is $f:\bbn\ra\bbn$ monotone increasing such that $\rz f(\om/2)\geq 2\Om$ and define $m:\bbr_+\ra\bbr_+$ on small values by $m(1/j)=1/f(j)$ for $j\in\bbn$ and extend to strict monotone function (for sufficiently small values in $\bbr_+$) by interpolation. One can check that $\om/2<1/\d$ which implies that $\rz m(1/(\om/2))>\rz m(\d)$. Similarly, one can check that $2\Om>1/\k$. Also, by the definition of $f$ and $m$, $\rz m(1/(\om/2))=1/f(\om/2)\leq 1/(2\Om)$. Stringing the three previous inequalities together gets $\rz m(\d)<\k$, a contradiction.
\end{proof}

    Given this, let's give a criterion for $[f]\in\SG$ to be a continuous germ.
    We first need a preparatory abstract lemma that gives a (new) standard interpretation of incomparable infinitesimals.
\begin{lemma}\label{lem:  abstract k<<<d lemma}
    Suppose that $\om,\Om\in\rz\bbn$ are such that $\Om\ggg\om$, and let $r_j\in\bbr_+$ with $r_j\ra 0$ as $j\ra\infty$. Let $A:\bbr^n\x\bbr^n\ra\bbr_+$ and for $j\in\bbn$, let $S_A(j)$ denote the assertion:
\begin{align}
    \text{there is}\;r\in\bbr_+\;\text{such that}\;|x-y|<r\Rightarrow A(x,y)<r_j
\end{align}
    and for $n\in\bbn$, $\bbs_A(n,j)$ denote the statement $(x,y\in B_{r_n})\;\wedge\;S_A(j)$. Suppose that $\rz\bbs_A(\om,\Om)$ holds.
    Then there is $n_0\in\bbn$ such that $\bbs_A(n_0,j)$ holds for infinitely many $j\in\bbn$.
\end{lemma}
\begin{proof}
    If the conclusion does not hold, then for each $n\in\bbn$, there are only a finite number of $j\in\bbn$ such that $\bbs_A(n,j)$ holds. Therefore, for each $n\in\bbn$,  $L(n)\dot=\max\{j:\bbs_A(n,j)\;\text{holds}\}$ is a well defined integer. That is, $L:\bbn\ra\bbn$ is a map such that, by hypothesis and the definition of $L$, $\Om\leq\rz L(\om)$ ,  contradicting that $\om\lll\Om$.
\end{proof}
    We can now give our infinitesimal criterion for germ continuity. This will be suitable for the present work. We will give a second criterion in \autoref{subsec: top properties of germ composition}.
\begin{proposition}\label{prop: germ continuity from k<<<d}
     Suppose that $\d\in\mu(0)_+$ and $[f]\in\SG$ has strong $\d$-good numbers.
    Then $[f]\in\SG^0$.
\end{proposition}
\begin{proof}
    By hypothesis, there is $\k\lll\d$ such that if $\xi,\z\in \rz B_\d$ satisfy $|\xi-\z|$ is sufficiently small, then $|\rz f(\xi)-\rz f(\z)|<\k$. We will show that this implies that if $A(x,y)=|f(x)-f(y)|$, then there is $r,\om,\Om$ such that the hypothesis in \autoref{lem:  abstract k<<<d lemma} is satisfied for this quadruple $(r,\om,\Om,A)$.
    In fact, choosing $r_j=1/j$, $\om=2\brkfl{1/\d}$ and $\Om=\brkfl{1/\k}/2$ if $\brkfl{1/\k}$ is even or $(\brkfl{1/\k}-1)/2$ if $\brkfl{1/\k}$
    is odd, we get that $\om\lll\Om$ by \autoref{lem:d>>>k -> [1/d]<<<[1/k]}, and it's straightforward to verify that $\rz r_\om<\d$ and $\rz r_\Om\geq\k$. With these definitions and estimates, our hypothesis  implies that for $\xi,\z\in\rz B_{*r_\om}$ with $|\xi-\z|$ sufficiently small, we have $\rz A(x,y)<\rz r_\Om$ which is precisely the hypothesis of the previous lemma.
    It's conclusion therefore implies that there are $j_1,j_2,\ldots\in\bbn$, such that for each $k\in\bbn$ the following holds:
\begin{align}
    \text{there is}\;r\in\bbr_+\;\text{such that}\;|x|,|y|<r_{j_{n_0}}\;\text{and}\;|x-y|<r\Rightarrow |f(x)-f(y)|<r_{j_k}
\end{align}
    That is, since $r_{j_k}\ra 0$ as $k\ra\infty$, this says that for $\un{r}=r_{j_{n_0}}$, if  $x,y\in B_{\un{r}}$   intersected with the open set where the representative $f$ for $[f]$ is defined, we can, for any $k$,  make $|f(x)-f(y)|<r_{j_k}$ by choosing $|x-y|$ sufficiently small, ie.,  $f$ is (uniformly) continuous; eg., $[f]\in\SG^0$.
\end{proof}
    Before we move on, we wish to draw attention to the related construction in section 8.4.44 of Stroyan and Luxemburg, \cite{StrLux76}, that we discovered after the completion of the above material. Furthermore, the author would like to point out this chapter, in particular, as a remarkable resource for those looking for systematic nonstandard renditions of topological matters.

    Next, we want to draw a closer connection between values forcing continuity and the moduli we want to use for our topology. For this we need further development.

\begin{definition}\label{def: SN_delta families of numbers}
     If $0<r\in\bbr$ and $g:B_r\ra\bbr$, write $\mathbf{\norm{g}_r}\dot=\sup\{|g(x)|:x\in B_r\}$ so that $\mathbf{\rz\norm{g}}_{\boldsymbol{\d}}\dot=\rz\sup\{|\rz g(\xi)|:\xi\in B_\d\}$, we may write this as $\mathbf{\norm{g}}_{\boldsymbol{\d}}$.
     Let $\boldsymbol{\wh{\SN}_\d}$ denote the set $\{\rz\norm{\rz g}_\d:[g]\in\SG\}$ and $\boldsymbol{\SN_\d^0}=\{\rz\norm{\rz g}_\d:[g]\in\SG^0\}$.
     If $\boldsymbol{{}_m\wt{\SG}_0}$ denotes the set of $[g]\in\SG$ such that for sufficiently small $r_1,r_2\in\bbr_+$ with $r_1<r_2$, we have $\|g\|_{r_1}\leq\|g\|_{r_2}$ and $\|g\|_r\ra 0$ as $r\ra 0$, we will let $\boldsymbol{\wt{\SN}_\d}$ denote $\{\rz\|g\|_\d:[g]\in {}_m\wt{\SG}_0\}$. If $\boldsymbol{{}_m\SG_0}$ denotes all $[g]\in{}_m\wt{\SG}_0$ satisfying $\|g\|_r<\|g\|_s$ for all sufficiently small positive $r<s$, we let ${\SN_\d}$ denote $\{\rz\|\rz g\|_\d:[g]\in{}_m\SG_0\}$.
     The germs in  ${}_m\wt{\SG}_0$ are said to be \textbf{pseudomonotone}.
     For $[f_0]\in\SG$, let $\boldsymbol{{}_m\wt{\SG}}_{\mathbf{[f_0]}}=\{[f+f_0]:[f]\in{}_m\wt{\SG}_0\}$ and ${}_m\SG_{[f]}\subset{}_m\wt{\SG}_{[f]}$ analogously.
\end{definition}
\begin{remark}
    Clearly, $\SN^0_\d,\SN_\d\subset\wt{\SN}_d$ and we shall see that for our topological purposes they are all equivalent. As we shall shortly see, the neighborhood subbase of the zero germ, $[0]$, that we will define will have the property that all such neighborhoods $U$ will satisfy $U\cap\SG=U\cap{}_m\wt{\SG}_0$ and so we shall see that, in our investigations of the behavior of this topology around $[0]$, the numbers $\wh{\SN}_\d$ will play no further role.
\end{remark}
    We need also to define some sets of monotone germs, our critical intermediaries for understanding the topology to be defined
\begin{definition}\label{def:  SM and SM^0}
    If $\bbr_+=\{t\in\bbr:t>0\}$ and $F(\bbr_+)$ consists of all functions $f:\bbr_+\ra\bbr_+$, then we have the following sets of germs.
 \begin{enumerate}
    \item    Let $\boldsymbol{\wt{\SM}}$ denote the set of germs, $[m]$, of $m\in F(\bbr_+)$ such that if $r,s$ are in the domain of $m$ with $r<s$, then $0<m(r)\leq m(s)$ and  $\lim_{t\ra 0}m(t)=0$.
    \item    Let $\boldsymbol{\SM}=\{[m]\in\wt{\SM}:\text{for all sufficiently small}\;r<t\in\bbr_+, m(r)<m(t)\}$.
    \item     Let $\boldsymbol{\SM^0}=\{[m]\in \SM:m\;\text{is continuous on some neighborhood of}\;0\}$.
 \end{enumerate}
   We will sometimes write this as $[m]\in\SM$, or speak of the germ of an element $m\in\SM$ at $0$ and identify this $[m]\in\SM$ with $\rz m|\mu(0)_+$, ie., the transfer of $m$ restricted to the positive infinitesimals (which by earlier arguments uniquely define $[m]$).
   If $\Fr\in\wt{\SN}_\d$, $\wt{\SN}^\Fr_\d$ will denote those $\Fs\in\wt{\SN}_\Fr$ with $\Fs<\Fr$.
\end{definition}

\begin{remark}\label{rem: e<<<d and f(d)<e -> f(d)=0}
    Clearly we have that $\SM^0\subsetneqq\SM\subsetneqq\wt{\SM}$.
    Also note that if $[m]\in\SM^0$, then $[m^{-1}]\in\SM$, where here $m^{-1}\in[m^{-1}]$, the compositional inverse of $m$, may be defined on an arbitrarily small (deleted) neighborhood of $0$ in $\bbr_+$.
    Note that if $f\in F(\bbr_+)$ has values in $[0,\infty)$, $\d\in\mu(0)_+$ and $\e\lll\d$, then by definition,  $\rz f(\d)<\e$ implies, in fact, that $\rz f(\d)=0$.
    Further, it's clear that if $[m]$ is the germ of a monotone increasing function on $\bbr_+$ with values in $[0,\infty)$ that satisfies $\rz m(\Ft)=0$ for some $\Ft\in\mu(0)_+$, then $[m]\not\in\wt{\SM}$ as $\rz m(\Ft)=0$ clearly implies $[m]=0$. Therefore, we have that if $\d,\e\in\mu(0)_+$ with $\e\lll\d$, then $[m]\in\wt{\SM}$ implies that $\rz m(\d)>\e$.
\end{remark}

    $\SM,\SM^0$ and $\wt{\SM}$ will play the role of intermediaries in working with the topology we will shortly define. We will also shortly see that, from the perspective of the moduli for our topology, they all are essentially the same, although for technical reasons, all will play a role. Basically, our topology will be defined in terms of the internal norms $\rz\!\norm{g}_\d$ for some fixed $\d\in\mu(0)_+$ and arbitrary $[g]\in\SG$. But we will also need to know that this topology is independent of $\d$. So in looking for a good set of values for these *norms, we will therefore be led to consider germs of functions of the form $m(t)\dot=\|g\|_t$ easily seen to be elements of $\wt{\SM}$ or its subsets.
    Then if we choose, eg., one of the sets $\SN_\d$ for our set of ``norm'' moduli and if $\Fr\in\SN_\d$ and $[g]\in\SG^0$ has $\rz\|g\|_\d<\Fr$, then in fact $[g]\in{}_m\SG_0$, ie., $t\mapsto \|g\|_t$ is an element of $\wt{\SM}$. This is our motivational perspective.

    We will often need to consider the set of values of families of monotone germs at a given infinitesimal and so the following notation will be useful.
\begin{definition}\label{def:{}_dSM, {}_dwtSM, etc}
    Given $\d\in\mu(0)_+$, let $\boldsymbol{{}_\d\wt{\SM}}=\{\rz m(\d):[m]\in\wt{\SM}\}=\{\norm{\rz m}_\d:[m]\in\wt{\SM}\}$. Similarly, define the subsets $\boldsymbol{{}_\d\SM}$ and $\boldsymbol{{}_\d\SM^0}$ of ${}_\d\wt{\SM}$. Generally, if $\SB\subset\wt{\SM}$, let $\boldsymbol{{}_\d\SB}\subset{}_\d\wt{\SM}$ denote the set $\{\rz f(\d):[f]\in\SB\}$.
\end{definition}

\begin{lemma}\label{lem: SN_d=SM_d}
     We have ${}_\d\wt{\SM}=\wt{\SN}_\d$ and similarly ${}_\d\SM=\SN_\d$ and ${}_\d\SM^0=\SN^0_\d$.
\end{lemma}
\begin{proof}
    Just note that if $[f]\in\SG$, and $f\in[f]$ is any representative with $\norm{\rz f}_\d\sim 0$, then on the neighborhood where it's defined $t\mapsto m(t)=\norm{f}_t\in\SM$  so that $\rz m(\d)=\norm{\rz f}_\d$ giving $\SN_\d\subset\SM_\d$. On the other hand, if $[m]\in\SM$, with a representative $m$, then $f(x)= m(|x|)\in\un{F}(n,1)_0$ with $\norm{\rz f}_\d=\rz m(\d)\sim 0$.
\end{proof}
    The following simple technical fact will be useful in our arguments.
\begin{lemma}\label{lem: la<Fr all Fr in Nd->la<<<Fr,also}
    Suppose that $\la\in\mu(0)_+$ satisfies $\la<\Fr$ for all $\Fr\in\SN_\d$.Then $\la\lll\Fr$ for all $\Fr\in\SN_\d$, eg., $\la\lll\d$.
    Also, if $\k,\d\in\mu(0)_+$ with $\k\lll\d$ and $\Fr\in\SN_\d$, then $\k\lll\Fr$.
\end{lemma}
\begin{proof}
    This follows from the definitions. Suppose that $\la\not\lll\Fr_0$ for some $\Fr_0\in\SN_\d$. That is, there is $[m]\in\SM$ such that $\rz m(\Fr_0)\leq\la$. But, by definition $\Fr_0=\rz m_0(\d)$ for some $[m_0]\in\SM$ and therefore $\rz (m\circ m_0)(\d)\leq \la$, a contradiction, as $[m\circ m_0]\in\SM$.
    The second assertion's proof is similar: if not, then there is $[m]\in\SM$ with $\k\geq \rz m(\Fr)$ and using that $\Fr=\rz m_0(\d)$ for some $[m_0]\in\SM$, then we have $\k\geq\rz(m\circ m_0)(\d)$ which says $\k\not\lll\d$, a contradiction.
\end{proof}
\begin{definition}\label{def: good coinitial subsets}
    If $[f]\in\SG$, we say that $\mathbf{[f]}$ \textbf{has good coinitial subsets for} $\boldsymbol{\d}$ if the following holds. There is $\la\lll\d$ such that for each $\Fs\in\SN_\d$, if $\xi,\z\in\rz B_\d$ satisfy $|\xi-\z|<\la$, then $|\rz f(\xi)-\rz f(\z)|<\Fs$.
\end{definition}
\begin{corollary}\label{cor: [f] good coin subsets->f is C^0}
    Suppose that $[f]\in\SG$ has good coinitial subsets for $\d$. Then $[f]\in\SG^0$.
\end{corollary}
\begin{proof}
    By hypothesis, there is $\la\in\mu(0)_+$ such that for each $\Fs\in\SN_\d$
  \begin{align}
    \rho\;\dot=\rz\sup\{|\rz f(\xi)-\rz f(\z)|:\xi,\z\in B_\d\;\text{and}\;|\xi-\z|<\la\}<\Fs
  \end{align}
    and so, by the above lemma $\rho\lll\d$. But then, by \autoref{prop: germ continuity from k<<<d} above, $[f]\in\SG^0$.
\end{proof}
   We now can give a definition and prove a result that will be critical to the convergence properties of the topology we will define in the next section.
\begin{definition}\label{def: U_Fr}
     Given $\Fr\in\SN_\d$, let $\mathbf{U}_{\boldsymbol{\Fr}}=\mathbf{U}^{\boldsymbol{\d}}_{\boldsymbol{\Fr}}\subset\SG$ denote the  set $\{[f]\in\SG:\norm{\rz f}_\d<\Fr\}$.
\end{definition}
   We will now begin to use nets in this text. In the next section it will become clear why they are necessary. A modern thorough treatment of nets in topology and analysis given by eg., Howes, see \cite{Howes1995}. When we define our topology in the next section, we will give a sufficient discussion of nets in the context of coinitial subsets of families of infinitesimals.
   Given the preliminaries above, we can now prove a result that will be important in the convergence properties of our topology for continuous germs.
\begin{proposition}\label{prop: [fd]->[g]--> [g] is cont}
   Suppose that $[g]\in\SG$ and $([f_d]:d\in D)$ is an upwardly directed net (see the text before \autoref{def: of coinitial partial order relations}) in $\SG^0$ with the property that for each $\Fr\in\SN_\d$, there is $d_0\in D$ such that if $d>d_0$, then $[f_d-g]\in U_\Fr$. Then $[g]\in\SG^0$.
\end{proposition}
\begin{proof}
    This will be just a 3 epsilon argument in a new environment. By  \autoref{cor: [f] good coin subsets->f is C^0} above, we need to show that for some $\la\in\mu(0)_+$, $\la\lll\d$, the following holds. If $\Fs\in\SN_\d$, then for $\xi,\z\in B_\d$ with $|\xi-\z|<\la$, we have $|\rz g(\xi)-\rz g(\z)|<\Fs$. But as $[f_d]\in\SG^0$, we have that if $0<\la\lll\d$, then for $\xi,\z\in B_\d$ with $|\xi-\z|<\la$, we have $|\rz f_d(\xi)-f_d(\z)|\lll\d$. So if $\Fr\in\mu(0)_+$ with $\Fr<\Fs/3$ and if we choose $d_0\in D$ so that $d>d_0$ implies $\|\rz f_d-\rz g\|_\d<\Fr$ (which by hypothesis can be done), then for $\xi,\z\in B_\d$ with $|\xi-\z|<\la$, and $d>d_0$ we have (leaving off the *'s)
  \begin{align}
    |g(\xi)-g(\z)|\leq |g(\xi)-f_d(\xi)|+|f_d(\xi)-f_d(\z)|+|f_d(\z)-g(\z)|.
  \end{align}
    The result follows as all terms are less than $\Fs/3$.
\end{proof}

   We can now proceed to our topology and its properties.
\section{Topology on $\SG$}\label{sec: topology-fixed target}
\subsection{Definition of $\tau$ and convergence properties}
    Given the work done up to now, we will give the obvious definition for our topology on $\SG$ and then begin to develop its properties. We will finish the proof that a convergent net of continuous germs is a continuous germ, \autoref{thm: convergnet of cont germs in SG converges to cont}.  We will prove that $\SG^0$ and its higher dimensional analogs have good topological algebraic properties in the next section,  \autoref{subsec: top properties ring structure} and \autoref{subsec: top properties of germ composition}. In this section, we prove that our moduli do not have countable coinitial subsets, \autoref{prop: SPSL has no countable coin subset}. We prove that the topology defined in terms of the given infinitesimal $\d$, ie., $\tau^\d$, is independent of the choice of infinitesimal, see eg., \autoref{cor: tau^d_1 = tau^d_2}.
   Finally, we prove that any good topology on $\SG$ cannot be first countable, see \autoref{thm: pzT has germ scales->not 1st countable}.


    For a given $0<\d\sim 0$, using $\SN_\d$  (or equivalently $\SN^0_\d$ as we shall see) we will now define our topology. Recall that $U^\d_\Fr=\{[g]\in\SG:\rz\|g\|_\d<\Fr\}$.
\begin{definition}\label{def: of topology  at 0 germ}
       Let $\pmb{\tau}_0=\pmb{\tau}^{\boldsymbol{\d}}_0=\{U^\d_\Fr:\Fr\in\SN_\d\}$ (the subscript $0$ indicates neighborhoods of the zero germ $[0]$).  If $\d$ is fixed in the discussion, we will often write $\tau$ or $\tau_0$ and $U_\Fr$ leaving off the $\d$'s.
     Given $\tau_0$ above and $[f]\in\SG$, let $\boldsymbol{\tau}_{\mathbf{[f]}}$ denote the $[f]$ translation of $\tau_0$; ie., $\tau_{[f]}=\{U+[f]:U\in\tau_0\}$, $U+[f]$ being $\{[g+f]:[g]\in U\}$ and let $\boldsymbol{\tau=\tau^\d}$ denote the topology generated, in the usual way, by finite intersections of arbitrary unions of elements of $\tau_{[f]}$ as $[f]$ varies in $\SG$. In particular, $\tau^\d_0$ is a subbase of the neighborhoods of $[0]$ in $\tau^\d$.
\end{definition}
    Our approach to the topology will be in terms of convergence.
    Therefore, assuming a familiarity with nets for the moment (see the next part), we therefore have the obvious criterion for germ convergence in this topology.
\begin{proposition}
     Suppose that $(D,<)$ is an upwardly directed set and that $([f_d]:d\in D)$ is a $D$ net in $\SG$. Then $([f_d]:d\in D)$ converges in $\tau^\d$ to the zero germ, $[0]$,  in $\SG$ if for each $\Fr\in\SN_\d$, there is $d_0\in D$ such that if $d\in D$ with $d>d_0$, then  $[f_d]\in U^\d_\Fr$, ie., $\norm{\rz f_d}_\d<\Fr$. We have that $([f_d]:d\in D)$ converges to $[g]\in\SG$ if and only if $([f_d-g]:d\in D)$ converges to $[g]$.
\end{proposition}
     As the topology has this translation invariance, we will generally be concerned with (questions of) net convergence to $[0]$.

     From the remarks after the definition of the $\SN_\d$'s, we have that for any $\Fr\in\SN_\d$ that $U^\d_\Fr\cap\SG\subset {}_m\SG_0$. This says that when working in the neighborhoods $U_\Fr$, it suffices to consider only elements of ${}_m\SG_0$ and therefore elements of $\wt{\SM}$ when considering the $\d$-norms.

\begin{proposition}
     $(\SG,\tau^\d)$ is absolutely convex, Hausdorff  vector space that is not absorbent (hence not locally convex) and scalar multiplication is not continuous.
\end{proposition}
\begin{proof}
    We have that the elements of the subbase at $[0]$, the $U_\Fr$'s for $\Fr\in\SN_\d$ are clearly convex and balanced, and so absolutely convex.  But the topology is  certainly not absorbent; for given $\Fr\in\SN_\d$, there is $\Fs\in\SN_\d$ with $t\Fr<\Fs$ for all $t\in\bbr_+$, ie., because $\SN_\d$'s total order is non-Archimedean. Similarly, if $r_j\in\bbr_+$ for $j\in\bbn$ with $r_j\ra 0$ as $j\ra\infty$ and $[f]\in\SG\ssm\{[0]\}$, then $r_j[f]\not\ra[0]$ in $\tau^\d$, for if $\rz\|f\|_\d=\Fr\in\SN_\d$, then there is $\Fs\in\SN_\d$ such that $\rz r_j\rz\|f\|_\d=\rz r_j\cdot\Fr>\Fs$ for all $j\in\bbn$.
\end{proof}

        The following theorem along with the nondiscrete nature of our topology indicates that our topology has some good properties; in particular this result indicates that $\tau^\d$ convergence is analogous to uniform convergence.
\begin{theorem}\label{thm: convergnet of cont germs in SG converges to cont}
    Suppose that $D$ is a directed set and that $d\in D\mapsto [f^d]$ is a $D$-net in $\SG^0$ that is $\tau_0^\d$ convergent to $[g]\in\SG$. Then $[g]\in\SG^0$.
\end{theorem}
\begin{proof}
    Clearly, this is just a restatement of  \autoref{prop: [fd]->[g]--> [g] is cont}.
\end{proof}

    On the other hand, let's note some curious  properties of this topology.
\begin{corollary}\label{cor: C^0 germs closed nowher dense}
    We have ${}_m\SG_0=\cup\{U_\Fr:\Fr\in\wt{\SN}_\d\}$ and $\SG=\cup\{{}_m\SG_{[f]}:[f]\in\SG\}$. If $[f],[g]\in\SG$ with $\rz\|f-g\|_\d\ggg\d$, then ${}_m\SG_{[f]}\cap{}_m\SG_{[g]}$ is empty. (Note that, with sufficient saturation, $\rz\|f-g\|_\d\ggg\d$ is not an empty condition.)
    If $[f]\in\SG$ with $\rz\|f\|_\d\ggg\d$, then $[f]\not\in\SG^0$. For every $\Fr\in\SN_\d$, there is a discontinuous germ $[g]\in U_\Fr$. If $[g]\in\SG$ is a discontinuous germ, then there is a $\tau$ neighborhood of $[g]$ consisting of discontinuous germs. In particular, $\SG^0$ is nowhere dense in $\SG$.
\end{corollary}
\begin{proof}
    We will just verify the last three assertions of which the last follows from the previous two and \autoref{thm: convergnet of cont germs in SG converges to cont}. Let's verify that if $\Fr\in\SN_\d$, then there is $[f]\in\SG$ such that $[f]\in U_\Fr$ that is not continuous. First, note that if $[m]\in\SM^0$, with $\rz\|m\|_\d<\Fr$, then defining $[\wt{m}]$ to be zero on $\rz\bbq_+$ and $[m]$ on $\rz\bbr_+\ssm\rz\bbq_+$, it's  clear that $\rz\|\wt{m}\|_\d<\Fr$ but $[\wt{m}]\not\in\SM^0$. Finally, define $[f]$ to be germ at $0$ of the map $x\mapsto\wt{m}(|x|)$.  Next, let's verify that if $[g]\in\SG\ssm\SG^0$, then there is $\Fr\in\SN_\d$ such that $U^\d_\Fr([g])$ consists of discontinuous germs. Suppose not, then for every $\Fr\in\SN_\d$, there is $[f_\Fr]\in\SG^0\cap U^\d_\Fr([g])$. But this says that $[g]$ is a limit point of the net $([f_\Fr]:\Fr\in\SN_\d)$ and so as $\SG^0$ is closed in $\SG$, by \autoref{thm: convergnet of cont germs in SG converges to cont}, then $[g]\in\SG^0$, a contradiction.
\end{proof}

    {\it We will not spend further effort on the basic topological properties of $\tau^\d$ until we consider algebraic properties in the next section. In this section, our efforts will be focused on (a) proving that $\tau^\d$ is independent of $\d$ and (b) on the non-first countability of $\tau$ (or any good topology on $\SG$!).} To these ends, we begin with motivation for further  definitions.

     There are two ways in which to define bounding conditions on sets of germs in $\wt{\SM}$. The most obvious is to follow the definition of $U_\Fr$ as $\Fr$ varies in $\wt{\SN}_\d$ down to the germ norm space $\wt{\SM}$ and look at those elements $[m]\in\wt{\SM}$ that are bounded above (or below) by some $\Fr_0\in\wt{\SN}_\d$ at $\d$. The second is to bound uniformly over all infinitesimals $\d\in\mu(0)_+$, ie., consider all those $[m]\in\wt{\SM}$ that are bounded above (or below) by a given $[m_0]\in\wt{\SM}$. Surprisingly, these two modes for bounding are asymptotically equivalent for our topology (the import of \autoref{prop: ptwise bndd subset unif bndd}). This is a critical fact for most considerations here. Let's define these two families of bounding sets.
\begin{definition}\label{def: SU([m]), SM([m]), M^Fr, etc}
    Generally, if $\SB\subset\wt{\SM}$, $[m_0]\in\wt{\SM}$ and $\Fr_0\in\wt{\SN}_\d$, we need to define four associated sets. Let $\boldsymbol{\SB}\mathitbf{([m_0])_u}$, respectively $\boldsymbol{\SB}\mathitbf{([m_0])}_{\pmb{\ell}}$, denote the set of $[m]\in\SB$ satisfying $\rz m(\la)<\rz m_0(\la)$ for all $\la\in\mu(0)_+$, respectively $\rz m(\la)>\rz m_0(\la)$. Let $\boldsymbol{\SB^{\Fr_0}}$, respectively $\boldsymbol{\SB_{\Fr_0}}$, denote the set of set of $[m]\in\SB$ satisfying $\rz m(\d)<\Fr_0$, respectively $\rz m(\d)>\Fr_0$. These last two subsets of $\SB$ may sometimes be denoted by $\boldsymbol{{}_\d\SB^{\Fr_0}}$, respectively $\boldsymbol{{}_\d\SB_{\Fr_0}}$, to indicate the dependence on the choice of the infinitesimal $\d$.
     As it is most commonly used, we will often write $\boldsymbol{\SB}\mathitbf{([m])}$ for $\SB([m])_u$.

     With respect to subsets of $\SG$, if $[\un{m}]\in\wt{\SM}$, let
  \begin{align}
    \boldsymbol{\SU}\mathitbf{([\un{m}])}=\SU([\un{m}])_u=\{[g]\in\SG:\rz\|g\|_\la<\rz\un{m}(\la)\;\text{for}\;\la\in\mu(0)_+\}
  \end{align}
    and occasionally, we will use the corresponding lower bound set $\boldsymbol{\SU}\mathitbf{([\un{m}])}_{\pmb{\ell}}$.
\end{definition}
\begin{remark}
    Our $\SB$ will generally be one of $\wt{\SM},\SM$ or $\SM^0$, hence we have twelve associated sets. $\SB$ may also be a subset of special affine germs, $\SP\SL^0$ yet to be defined (\autoref{def: PL^0 germs at 0}).
\end{remark}
    Given these preliminaries, let's formalize our connection between the topology $\tau^\d$ defined on $\SG$ and subsets of $\wt{\SM}$.

\begin{definition}\label{def: FL:SG->SM}
    For $[g]\in{}_m\wt{\SG}_0$, define $\pmb{\FL}:{}_m\wt{\SG}\ra\wt{\SM}$ as follows. $\FL([g])\in\wt{\SM}$ is defined to be the germ at $0$ of the map $t\mapsto\|g\|_t$.
    If $\SB\subset\wt{\SM}$, then $\FL(\SB)$ denotes the set of $\FL([f])$ for $[f]$ in $\SB$.
\end{definition}
    Note, then, that $\FL: {}_m\wt{\SG}_0\ra\wt{\SM}$ is a surjective map so that for a given $\Fr\in\wt{\SN}_\d$ and  $[g]$ in ${}_m\wt{\SG}_0$, we have that $\FL([g])(\d)<\Fr$ if and only if $[g]\in U_\Fr$.
    With the notation introduced, then we clearly have the following statement: for each $\Fr\in\wt{\SN}_\d$, $\FL:{}_m\wt{\SG}\ra\wt{\SM}$ satisfies $\FL(U_\Fr)=\wt{\SM}^\Fr$  and $\FL^{-1}(\wt{\SM}^\Fr)=U_\Fr$. In particular, the $U_\Fr$'s are totally determined by the set of $\wt{\SM}^\Fr$'s.
    Similarly, with the above definitions, we  have the correspondence: $\FL(\SU([\wh{m}]))=\wt{\SM}([\wh{m}])$ and $\FL^{-1}(\wt{\SM}([\wh{m}]))=\SU([\wh{m}])$. Let's record this simple but critical simplification as a lemma.
\begin{lemma}\label{lem: FL: SG >->> SM}
    We have $\FL({}_m\wt{\SG}_0)=\wt{\SM}$, $\FL({}_m\SG_0)=\SM$ and $\FL({}_m\SG^0_0)=\SM^0$.
    Also we have that $\FL$ gives an order preserving bijection
   \begin{align}
     \FL:\{U_\Fr:\Fr\in\SN_\d\}\ra\{\wt{\SM}^\Fr:\Fr\in\SN_\d\}
   \end{align}
      and also an order preserving bijection
   \begin{align}
     \FL:\{\SU([m]):[m]\in\wt{\SM}\}\ra\{\wt{\SM}([m]):[m]\in\wt{\SM}\}.
   \end{align}

\end{lemma}
     Once we have some more terminology on order and convergence, we will see that a simple consequence of part of this statement will often allow us to reduce convergence (to the $[0]$ germ) in $\SG$ to convergence properties of the corresponding subsets of $\wt{\SM}$ under the correspondence defined by $\FL$.

      We will return to these sets and their variations in later sections; here we have a simple but quite useful statement.

\begin{lemma}\label{lem: wtSM_r sub wtSM(m)->U_r sub U(m)}
    Suppose that $\Fr_0\in\wt{\SN}_\d$ and that $[m_0]\in\wt{\SM}$ and further that $\wt{\SM}^{\Fr_0}\subset\wt{\SM}([m_0])$. Then $U_{\Fr_0}\subset \SU([m_0])$.
\end{lemma}
\begin{proof}
    The proof is just unwinding the definitions using the previous lemma. If $[g]\in U_{\Fr_0}$, then by the above remarks $[g]\in \FL^{-1}([m])$ for some $[m]\in\wt{\SM}^{\Fr_0}$ and so by hypothesis, $[m]\in\wt{\SM}([m_0])$. But again, by the above lemma, $\FL^{-1}([m])\subset\SU([m_0])$ and we are finished as $[g]\in \FL^{-1}([m])$.
\end{proof}

    Here, as in most of this paper, the infinitesimal $\d$ will be fixed and implicit, and so here we will write $U_\Fr$ for $U^\d_\Fr$.

    Before we can say anything more about this topology we need more formalities on orders. See Fuchs, \cite{Fuchs1963}, for a good coverage of the mathematics of ordered algebraic systems.
    Let's recall some basic notions and  definitions from the theory of ordered sets.
    Suppose that $(P,\leq)$ is a partially ordered set (ie., for all  $p,q,r\in P$ we have $p\leq p,p\leq q$ and $q\leq p$ implies $p=q$ and $p\leq q,q\leq r$ implies $p\leq r$) and $J\subset P$ with the induced partial order.
    We then say that $(P,<)$ is a (downward, respectively upward) set, if for each $p,q\in P$, there is $r\in P$ such that $r<p$ and $r<q$, respectively $r>s$ and $r>q$.
\begin{definition}\label{def: of coinitial partial order relations}
    If $(P,<)$ is a partially ordered, downward directed set and $J\subset P$, then we say that $\mathbf{J}$\textbf{ is coinitial in }$\mathbf{P}$ with respect to $<$, if for all $p\in P$, there is $a\in J$ such that $a\leq p$.
    Suppose that we have two subsets $J,K\subset P$. Then we say that $\mathbf{J}$ \textbf{is coinitial with} $\mathbf{K}$, if for all all  $k\in K$, there is $j\in J$ such that $j\leq k$ and we say that $\mathbf{J}$ \textbf{and} $\mathbf{K}$ \textbf{are coinitial}, if $J$ is coinitial with $K$ and $K$ is coinitial with $J$.
    Suppose that $(T,<)$ is a totally ordered, upwardly directed, set with $S\subset T$ given the restricted total order. Suppose that $(D,<)$ is an upward directed set, and that $\SV=(\Fv_d:d\in D)$ is a $D$ net in $T$. Then we say that $\boldsymbol{\SV}$ \textbf{is convergently coinitial in the range of} $\mathbf{S}$, if for each $\Fs\in S$, there is $d_0\in D$ such that, if $d>d_0$ then $\Fv_d>\Fs$. If, in addition we have that for each $\Fv\in\SV$, there is $\Fs\in S$ with $\Fs\geq\Fv$, then we say that $\boldsymbol{\SV}$ \textbf{is convergently coinitial with} $\boldsymbol{S}$.
    If we speak about coinitiality when referring  to the partial order on germs, we will say \textbf{germwise coinitial}.
\end{definition}
    Before we proceed to our work with coinitial subsets of various subsets of $\wt{\SN}_\d$, let's give a useful convergence correspondence for our map $\FL$.
\begin{proposition}\label{prop: FD converg <-> FL(FD)converg}
    Let $\d\in\mu(0)_+$ and  suppose that $\FD=([f_d]:d\in D)$ is a net in ${}_m\wt{\SG}_0$ with $\FL(\FD)=\{\FL([f_d]):d\in D\}$ the corresponding net in $\wt{\SM}$. Then $\FD$ is convergent to the zero germ $[0]$ in $\tau^\d$ if and only if $\;\;{}_\d\FL(\FD)$ is convergently coinitial in $\wt{\SN}_\d$.
\end{proposition}
\begin{proof}
    The proof is a matter of unwinding definitions. First note, for a given $d\in D$  \autoref{lem: FL: SG >->> SM} implies $[f_d]\in U_\Fr$ if and only if $\FL([f_d])\in\wt{\SM}^\Fr$. But, by definition, this holds if and only ${}\;\;_\d\FL([f_d])<\Fr$. So the assertion follows from the above definition that a net $\FN=(\Fr_d:d\in D)$ in $\wt{\SN}_\d$ is convergently coinitial in $\wt{\SN}_\d$ if and only if for each $\Fs\in\wt{\SN}_\d$ there is $d_0\in D$ such that $d>d_0$ implies that $\Fr_d<\Fs$.
\end{proof}
     Here we introduce the subset $\SP\SL^0$ of $\SM^0$ with the following motivation. Although we show that it is germwise coinitial in $\wt{\SM}$, it is nevertheless a concretely defined subset of $\SM^0$ that is discretely defined. It also gives a sparse, rigid set of germs that nonetheless does not have a countable coinitial subset. In the next part, \autoref{subsec: ord preserv tops not countable}, a subset will be used as a prototypical example of the germ scale preservation phenomena. In \autoref{sec: standard interpret tau}, we will deploy it in one of our standard descriptions of $\tau$. Needless to say, there are many other `nice' subsets of $\SM^0$ that could play this role. In the first version of this work (see \cite{McGaffeyPhD}, chapter 7), a special set of rigid germs (power series described as ``Hardy fields'' by the author) played such a role.
\begin{definition}\label{def: PL^0 germs at 0}
    Let $\boldsymbol{\SP\SL^0}\subset\SM^0$ denote the set germs  $\pzp\in\SM^0$ mapping $1/\bbn=\{1/j:j\in\bbn\}$ into $1/\bbn$ that are piecewise affine. In particular, an element $[\pzp]\in\SM^0$ is in $\SP\SL^0$ if for all $j\in\bbn$ sufficiently large, $\pzp|_{[1/j+1,1/j]}$ is an affine map. If $0<\d\sim 0$, let $\boldsymbol{{}_\d\SP}\subset\mu(0)_+$ denote $\{\rz\pzp(\d):[\pzp]\in\SP\SL^0\}$.
\end{definition}

\begin{proposition}\label{prop: SPSL^0 coinitial in wt(SM)}
    $\SP\SL^0$ is a coinitial subset of $\wt{\SM}$; ie., for each  $[m]\in\wt{\SM}$ with $m(t)\not=0$ for all $t>0$, there is $[\wh{m}]\in\SP\SL^0$ such that $\wh{m}(t)\leq m(t)$ for all sufficiently small $t>0$.
    In particular, $\SM^0$ is germwise coinitial in $\wt{\SM}$.
\end{proposition}
\begin{proof}
    Let $[m]\in\wt{\SM}$ and $m\in[m]$. For $j\in\bbn$, let $A_j=\{t>0:m(t)<1/j\}$. Now, as $m$ is monotone (for $t$ small enough), we have that for all (sufficiently large) $j$ that $b\in A_j$ and $0<t<b$, implies $t\in A_j$. As $A_1\supseteq A_2\supseteq A_3\supseteq\cdots$ with $\cap\{A_j:j\in\bbn\}$ empty, there is $k_1<k_2<\cdots$ in $\bbn$ with $A_{k_j}\ssm A_{k_{j+1}}$ a nontrivial interval (ie., if $a,b\in A_{k_j}\ssm A_{k_{j+1}}$ and $a<t<b$, then $t\in A_{k_j}\ssm A_{k_{j+1}}$). So for each (sufficiently large) $j\in\bbn$, there is $a_j\in A_{k_j}\ssm A_{k_{j+1}}$ so that $a_j>a_{j+1}>\cdots$ with unique limit $0$. Given this, for each $j\in\bbn$ (sufficiently large), define $\wt{m}(a_j)=m(a_{j+2})$ and define $\wt{m}$ on $[a_{j+1},a_j]$ by affine interpolation between the value at $a_{j+1}$ and that at $a_j$.  With this, suppose that, for some sufficiently big $j$, $t\in (a_{j+1},a_j)$ so that, by definition $\wt{m}(t)<\wt{m}(a_j)=m(a_{j+2})<1/k_{j+2}$ as $a_{j+2}\in A_{k_{j+2}}\ssm A_{k_{j+3}}\subset A_{k_{j+2}}$. But $t\in (a_{j+1},a_j)$ implies eg., that $t>a_{j+1}$, ie., not in $A_{j+2}$ so that $m(t)\geq 1/k_{j+2}$. Piecing these inequalities together, we get $\wt{m}(t)<m(t)$, and so as $j$ is arbitrary in $\bbn$, and as, by definition $\wt{m}(a_j)<m(a_j)$,  we have a positive strictly monotone piecewise affine map $\wt{m}$ with $\wt{m}(t)<m(t)$ for all sufficiently small $t>0$.
    But, given such an $\wt{m}$, we can find $\wh{m}\in\SP\SL^0$ with $\wh{m}(t)<\wt{m}(t)$ for all sufficiently small $t>0$, as follows. Simply define, for sufficiently large $j$,
  \begin{align}
      \wh{m}(1/j)=\f{1}{\brkfl{\f{1}{\wt{m}(\f{1}{j+1})}}+1}.
  \end{align}
     where, as before, for $r\in\bbr_+$,
     $\brkfl{r}$
     is the least integer $k$ with $k\geq r$. Then, once again extending $\wh{m}$ to each interval $[1/j+1,1/j]$ (for sufficiently large $j\in\bbn$) by extending affinely from the endpoint values, clearly, $\wh{m}\in\SP\SL^0$. Furthermore, if for large $j$, $t\in (1/j+1,1/j)$, it's elementary to check that $\wh{m}(t)<\wt{m}(t)$.
\end{proof}

\begin{proposition}\label{prop: SPSL has no countable coin subset}
    Suppose that $[\pzp_1]>[\pzp_2]>\cdots$ is a sequence in $\SP\SL^0$. Then, there is $[\un{\pzp}]\in\SP\SL^0$ such that $[\un{\pzp}]<[\pzp_j]$ for all $j\in\bbn$.
    In particular, $\SN^0_\d$ does not have a countable coinitial subset.
\end{proposition}
\begin{proof}
    For each $k\in\bbn$, define
  \begin{align}
    \un{\pzp}(1/k)=\f{1}{\max\{1/\pzp_j(1/k):1\leq j\leq k\}+k}
  \end{align}
    and then define $\un{\pzp}$\; to be  an affine function on each interval $[1/(k+1),1/k]$ for each $k\in\bbn$. One can check that $[\un{\pzp}]$ has the required properties. In particular, as $[\un{\pzp}]<[\pzp_j]$ for all $j\in\bbn$, then $\rz\un{\pzp}(\d)<\rz\pzp_j(\d)$ for all $j$ and this along with the previous lemma gets the second assertion.
\end{proof}
\begin{remark}\label{rem: SN^-1 = SN -> no count cofinal}
   The following two facts imply that $\rz m(\d)\mapsto\rz m^{-1}(\d)$ as $[m]$ varies in $\SM^0$ gives an order reversing bijection $\SN^0_\d\ra\SN^0_\d$, and therefore coinitial subsets to cofinal subsets. In particular, the above proposition implies that $\SN^0_\d$ does not have a countable cofinal subset.
   First, we have $[m]\mapsto[m^{-1}]$ is a bijection on $\SM^0$. Second, if $[m_1],[m_2]\in\SM^0$ and $\e\in\mu(0)_+$ with $\rz m_1(\e)<\rz m_2(\e)$, then $\rz m_1(\e)>\rz m_2(\e)$.
   Furthermore, the switch map $(1/j,1/k)\mapsto(1/k,1/j)$ from $1/\bbn\x 1/\bbn$ to itself extends to a bijection $S:\SP\SL^0\ra\SP\SL^0$. This map does not send $[\pzp]$ to $[\pzp^{-1}]$, but, when evaluated at, eg., $\d=1/\om$ for some $\om\in\rz\bbn_\infty$, does send a coinitial subset of $\SN^0_\d$ to a cofinal subset of $\SN^0_\d$ (and so by \autoref{cor: B_d1 coinit<->B_d2 coinit} in the next section) for any $\d\in\mu(0)_+$.
\end{remark}
\subsection{$\tau$ is independent of choice of $\d$}\label{subsec: tau independ of delta}
    In this subsection, we will prove that $\tau^\d$ is independent of choice of $\d\in\mu(0)_+$. We do this by using nonstandard techniques to prove a series of infinitesimal pinching constructions, of which the next result will contain the primary pinching result. The topological independence of $\tau=\tau^\d$ from the choice of $\d$, \autoref{thm: tau^d on SM indep of d}, will follow from these pointwise implies uniform pinching results.

    If $[m]\in\SM$, recall the definitions of $\SM_\Fr$ and $\SM([m])$ and others in \autoref{def: SU([m]), SM([m]), M^Fr, etc}.
\begin{proposition}\label{prop: ptwise bndd subset unif bndd}
    Given $\d_0\in\mu(0)_+$, and $\Fr_0,\Fs_0\in\SN_{\d_0}$, there are $[\un{m}],[\ov{m}]\in\SM^0$ such that if $[m]\in\SM$ with $\rz m(\d_0)>\Fr_0$, then $[m]>[\un{m}]$ and if $\rz m(\d_0)<\Fs_0$, then $[m]<[\ov{m}]$.
    That is, $\SM_{\Fr_0}\subset\SM([\un{m}])_\ell$ and $\SM^{\Fs_0}\subset\SM([\ov{m}])_u$.
    Parallel statements hold for $\wt{\SM}$, ie., existence of $[\un{m}'],[\ov{m}']\in\SM$ with $\wt{\SM}_{\Fr_0}\subset\wt{\SM}([\un{m}'])_\ell$ and $\wt{\SM}^{\Fs_0}\subset\wt{\SM}([\ov{m}'])_u$.
\end{proposition}
\begin{proof}
    As both proofs are quite similar, we will prove the first and indicate the changes needed in the first proof to get the second assertion. Let $[m_0]\in\SM$ be such that $\Fr_0=\rz m_0(\d_0)$. Let $P$ denote the set of $k\in\bbn$ such that there is $r\in\bbr_+$ with $1/(k+1)<r\leq 1/k$ and $m(r)\leq m_0(r)$ for all $[m]\in\SM^{\Fr_0}$. Then there is $\om_0\in\rz\bbn_\infty$ such that $1/(\om_0+1)<\d_0\leq 1/\om_0$ with $\rz m(\d_0)<\Fr_0=\rz m_0(\d_0)$ for all $[m]\in\SM^{\Fr_0}$; eg., $\om_0\in\rz P$. Therefore, $P$ is infinite; and so if $k_1<k_2<\cdots$ is an enumeration of the integers in $P$, then for each $j\in\bbn$, there is $r_j\in\bbr_+$ with $1/(k_j+1)<r_j\leq 1/k_j$ satisfying $m(r_j)<m_0(r_j)$. Note, eg., that the $r_j$'s form a decreasing sequence with unique limit $0$. Given this, define $[\un{m}]\in\SM^0$ as follows. For $t>1/k_1$, define $\un{m}(t)>0$ arbitrarily monotone larger than $m_0(r_2)$. Otherwise, for each $j\in\bbn$, define $\un{m}(r_j)=m_0(r_{j+1})$ and as this definition gives $\un{m}(r_{j+1})<\un{m}(r_j)$ for all $j\in\bbn$, for each $j\in\bbn$, we can extend $\un{m}$ to a continuous monotonic function on the interval $[1/r_{j+1},r_j]$, giving an element $\un{m}\in\SM^0$. But then note that for each $[m]\in\SM^{\Fr_0}$, for sufficiently small $t>0$, $t\in[r_{j+1},r_j)$ for some $j\in\bbn$ and by definition
  \begin{align}
     m(t)\geq m(r_{j+1})>m_0(r_{j+1})=\un{m}(r_j)>\un{m}(t),
  \end{align}
     ie., $[m]\in\SM^{\Fr_0}$ implies that $[m]>[\un{m}]$, as we wanted.

     To prove the second assertion,  let $[m_0]\in\SM$ be so that $\Fs_0=\rz m_0(\d_0)$, and then define $P$ above again, except obviously now looking at those $k$ for which there is $r\in(1/(k+1),1/k]$  with $m(r)>m_0(r)$ for all $[m]\in\SM_{\Fs_0}$. The rest of the argument goes through with slight changes, noting now that we will now define $\ov{m}(r_j)=m_0(r_{j-1})$. As we are working with germs, we may begin  at $r_2$.

   The final statements for $\wt{\SM}$ follows from those for $\SM$ just proved and the fact that $\SM$ is germwise coinitial in $\wt{\SM}$.
\end{proof}

    In the following, if $[m_0]\in\wt{\SM}$, and $\SB\subset\wt{\SM}$, then $m_0\SB$ denotes $\{[m_0\circ m]:[m]\in\SB\}$.

\begin{proposition}\label{prop: un(SM)^r subset un(SM)(m)}
    Given $[\un{m}]\in \wt{\SM}$ and $\d\in\mu(0)_+$, there is $\Fr\in\SN_\d$ such that $\SM^\Fr\subset\SM([\un{m}])_u$.
    Similarly, given $[m]\in\wt{\SM}$, there is $\wt{\Fr}\in\SN_\d$ with $\wt{\SM}^{\wt{\Fr}}\subset\wt{\SM}([m])_u$.
\end{proposition}
\begin{proof}
    We will first prove both conclusions for $[\un{m}]\in\SM^0$.
    Suppose that the first conclusion does not hold; ie., there does not exist $\Fr\in\SN_\d$ with $\SM^\Fr\subset\SM([\un{m}])_u$  and let $\Fs=\rz\un{m}(\d)$. By  \autoref{prop: ptwise bndd subset unif bndd}, there is $[\wt{m}]\in\SM^0$ such that $\SM^\Fs\subset\SM([\wt{m}])_u$. Now as  representatives $\un{m},\wt{m}\in\SM^0$, both $\un{m}^{-1}$ and $\wt{m}^{-1}$ are in $\SM$ and so $\SM^\Fs\subset\SM([\wt{m}])_u$ implies that (leaving out some brackets)
\begin{align}
    \un{m}\circ\wt{m}^{-1}(\SM^\Fs)\subset \un{m}\circ\wt{m}^{-1}(\SM([\wt{m}])_u)=\SM([\un{m}])_u,
\end{align}
     since $\un{m}\circ\wt{m}^{-1}$ is (a representative of) the germ of a (monotone) homeomorphism. But, if $\wt{\Fs}\dot=\rz\un{m}^{-1}\circ\wt{m}(\Fs)$ (which is in $\SN_\d$ as $\Fs$ is), then again as $\un{m}\circ\wt{m}^{-1}$ is a (monotone) homeomorphism, we have that $\un{m}\circ\wt{m}^{-1}(\SM^{\Fs})=\SM^{\wt{\Fs}}$, contradicting our contrary conclusion.

     The second conclusion ($[\un{m}]$ still in $\SM^0$) follows from the statements concerning $\wt{\SM}$ in \autoref{prop: ptwise bndd subset unif bndd} just as the first conclusion did, once one realizes that, analogous to the above, if $\wh{m}\in\SM^0$, then $\wh{m}\wt{\SM}([m])_u=\wt{\SM}([\wh{m}\circ m])_u$.

     Now, if $[\un{m}]\in\wt{\SM}$, choose $[m_0]\in\SM^0$ with $[m_0]<[\un{m}]$ from which we get $\SM([m_0])_u\subset\SM([\un{m}])_u$ and $\wt{\SM}([m_0])_u\subset\wt{\SM}([\un{m}])_u$. Given this, the general case follows from the proof with $[\un{m}]\in\SM^0$.
\end{proof}
\begin{corollary}
    Given $\d_0\in\mu(0)_+$, $[\ov{m}]\in\SM$ and $\Fr_0=\rz\ov{m}(\d_0)\in\SN_{\d_0}$, there is $[\wh{m}]\in\SM$ such that
\begin{align}
    U^{\d_0}_{\Fr_0}\subset \SU([\wh{m}]).
\end{align}
\end{corollary}
\begin{proof}
    First of all,  \autoref{prop: ptwise bndd subset unif bndd} implies that
        $\SM^{\Fr_0}\subset\SM([\wh{m}])_u$ for some $[\wh{m}]\in\SM$.
    But then \autoref{lem: wtSM_r sub wtSM(m)->U_r sub U(m)} implies that
    $U_{\Fr_0}\subset\SU([\wh{m}])$.
\end{proof}
    We also have use for the reverse occurrence.
\begin{corollary}\label{cor: given m, U_r subset SU(m) some r}
    Given $\d\in\mu(0)_+$ and $[\un{m}]\in\SM$, there is $\un{\Fr}\in\SN_\d$ with $U_{\un{\Fr}}\subset\SU([\un{m}])$.
\end{corollary}
\begin{proof}
     This is an immediate consequence of  \autoref{prop: un(SM)^r subset un(SM)(m)} and \autoref{lem: wtSM_r sub wtSM(m)->U_r sub U(m)}.
\end{proof}

     The following direct corollary is critical to verifying that $\tau$ is independent of the choice of $\d$. (Recall the notations in \autoref{def:{}_dSM, {}_dwtSM, etc}.)
\begin{corollary}\label{cor: B_d1 coinit<->B_d2 coinit}
    Given positive $\d_1,\d_2\in\mu(0)_+$ and $\SB\subset\SM$, we have that ${}_{\d_1}\SB$ is coinitial in $\SN_{\d_1}$ if and only if ${}_{\d_2}\SB$ is coinitial in $\SN_{\d_2}$.
\end{corollary}
\begin{proof}
    Suppose that coinitiality does not hold,
    so that, without loss of generality, we have, say, ${}_{\d_1}\SB$ is not coinitial, but ${}_{\d_2}\SB$ is coinitial. But as ${}_{\d_1}\SB$ is not coinitial in $\SN_{\d_1}$, then there is $\Fr_1$ such that $\SB\subset\SM_{\Fr_1}$. Given this,  the \autoref{prop: ptwise bndd subset unif bndd} implies $\SB\subset\ov{\SM}([\wt{m}])_\ell$ for some $[\wt{m}]\in\SM$, ie., that for $[m]\in\SB$, we have $\rz m(\d)>\rz\wt{m}(\d)$ for all $\d\in\mu(0)_+$, which, in particular for $\d=\d_2$, implies that ${}_{\d_2}\SB$ is not cofinal in ${}_{\d_2}\SN$, a contradiction.
\end{proof}
    The previous fact will be useful, eg., see \autoref{prop: mult SM by elt SM is open}, but the following sharpening of it will be the critical fact in proving that $\tau^\d$ is independent of $\d$.
\begin{theorem}\label{thm: tau^d on SM indep of d}
    Suppose that we have an upwardly directed net $\pzE=\{[m_d]:d\in D\}\subset\SM$ and $\d_1,\d_2\in\mu(0)_+$. Then ${}_{\d_1}\pzE$ is convergently coinitial in $\SN_{\d_1}$ if and only if ${}_{\d_2}\pzE$ is convergently coinitial in $\SN_{\d_2}$.
    That is, $\pzE$ converges in $\tau_0^{\d_1}$ if and only if it converges in $\tau^{\d_2}_0$.
\end{theorem}
\begin{proof}
    In \autoref{cor: B_d1 coinit<->B_d2 coinit}, we have proved the coinitiality part; here we will prove converence. By definition, this means that at a given $\d\in\mu(0)_+$ we have that for each $\Fr\in\SN_\d$, there is $d_0\in D$ such that $d>d_0$ implies that $\rz m_d(\d)<\Fr$. We must prove that this condition holds at $\d_1$ if and only if it holds at $\d_2$, and by symmetry it suffices to show that if it holds at $\d_1$, then it holds at $\d_2$. To the contrary, suppose that it holds at $\d_1$, but not at $\d_2$. Therefore, there is $\Fr_0\in\SN_{\d_2}$ such that
    for every $d_0\in D$,
  \begin{align}
    \wt{\pzE}(d_0)\;\dot=\{[m_d]\in\pzE:d>d_0\;\text{and}\;\rz m_d(\d_2)>\Fr_0\}
  \end{align}
    is nonempty. By  \autoref{prop: un(SM)^r subset un(SM)(m)}, there is $[\un{m}]\in\SM^0$ such that if $[m]\in\SM$ has $\rz m(\d_2)>\Fr_0$, then $\rz m(\d)>\rz\un{m}(\d)$ for all $\d\in\mu(0)_+$. In particular, for an arbitrarily given $d\in D$, we have $[m_d]\in\wt{\pzE}(d)$ implies that $\rz m_d(\d_1)>\rz\wt{m}(\d_1)$. That is, if $\Fs_0\dot=\rz\wt{m}(\d_1)\in\SN_{\d_1}$, then, as $\wt{\pzE}(d)$ is nonempty for all $d\in D$, we have for all $d_0$, that there is $d\geq d_0$ with $\rz m_d(\d_1)>\Fs_0$, a contradiction.
\end{proof}
   The above theorem along with the convergence correspondence between nets in ${}_m\wt{\SG}_0$ and the corresponding ones in $\wt{\SM}$ immediately implies the following.
\begin{corollary}\label{cor: tau^d_1 = tau^d_2}
    Suppose that $\d_1,\d_2\in\mu(0)_+$ and that $\FD=\{[f_d]:d\in D\}$ is a net in $\SG$. Then $\FD$ converges to $[0]$ in $({}_m\wt{\SG}_0,\tau^{\d_1}_0)$ if and only if it converges to $[0]$ in $({}_m\wt{\SG}_0,\tau^{\d_2}_0)$.
\end{corollary}
\begin{proof}
    By \autoref{prop: FD converg <-> FL(FD)converg}, $\FD$ converges to $[0]$ in $({}_m\wt{\SG}_0,\tau^{\d_1}_0)$ if and only if$\!\!{}\;_{\d_1}\FL(\FD)$ is convergently coinitial in $\wt{\SN}_{\d_1}$ which by the above \autoref{thm: tau^d on SM indep of d} is equivalent to$\!\!{}\;\;_{\d_2}\FL(\FD)$ being convergently coinitial in $\wt{\SN}_{\d_2}$ which again by \autoref{prop: FD converg <-> FL(FD)converg} is equivalent to saying that $\FD$ converges to $[0]$ in $({}_m\wt{\SG}_0,\tau^{\d_2}_0)$.
\end{proof}

\subsection{Order preserving topologies aren't first countable}\label{subsec: ord preserv tops not countable}
    Recall that ${}_\d\SP\subset\SN_\d$ denotes the set of $\rz\pzp(\d)$ for $[\pzp]\in\SP\SL^0$. So ${}_\d\SP$ has the induced total order coming from the (transferred) order on  $\rz\bbr$ and we know from  \autoref{prop: SPSL has no countable coin subset} above that ${}_\d\SP$ does not have a countable coinitial subset with resect to this order.
    Now from \autoref{prop: ptwise bndd subset unif bndd} above, we have that, for each $\g\in{}_\d\SP$, there is $\mathbf{[\pmb{\pzq}_{\boldsymbol{\g}}]}\in\SP\SL^0$ with the property that if $[\pzp]\in\SP\SL^0$ satisfies $\rz\pzp(\d)>\g$, then $[\pzp]>[\pzq_\g]$. In particular, if $\g_0,\a,\a_0\in\SN_\d$ with  $\a_0>\a$ and $\a_0$ is sufficiently smaller than $\g_0$ to get $\pzq_{\g_0}(\d)>\a_0$ and so eg., $\pzq_{\g_0}(\d)>\a$, then $[\pzq_{\g_0}]>[\pzq_\a]$ by the definition of $[\pzq_{\a}]$.
    Given this, let
  \begin{align}
    \pmb{\pzQ}=\{[\pzq_\g]:\g\in{}_\d\SP\}.
  \end{align}
    Suppose that $[\pzq_{\g_1}]>[\pzq_{\g_2}]>\cdots$ is any strictly decreasing sequence in $\pzQ$. Then $\rz\pzq_{\g_1}(\d)>\rz\pzq_{\g_2}(\d)>\cdots$ is a strictly decreasing sequence in ${}_\d\SP$ and so by the above proposition (\autoref{prop: SPSL has no countable coin subset}), there is $\a\in{}_\d\SP$ with $\rz\pzq_{\g_j}(\d)>\a$ for all $j$, and so by definition, $[\pzq_{\g_j}]>[\pzq_\a]$ for all $j\in\bbn$. We have proved the first half of the following.
\begin{proposition}\label{prop: (P^d,SQ) asymp tot ord, uncount}
    For each $\g_0\in{}_\d\SP$, there is smaller $\a_0\in{}_\d\SP$ such that if $\a\in{}_\d\SP$ is such that  $\a_0>\a$, then $[\pzq_{\g_0}]>[\pzq_\a]$.
    Furthermore, $\pzQ$ does not have a countable coinitial subset with respect to germ order.
    Finally, given $\g_0$ in ${}_\d\SP$, there is larger $\un{\g}\in{}_\d\SP$ such that $\g>\g_0$ implies that $[\pzq_\g]>[\pzq_{\un{\g}}]$.
\end{proposition}
\begin{proof}
    To prove the last assertion, first note that by definition $[\pzq_{\g_0}]\in\SM_{\g_0}$ and so by \autoref{prop: ptwise bndd subset unif bndd} there is $[\wt{m}]\in\SM$ such that $\g>\g_0$ implies that $[\pzq_\g]>[\wt{m}]$. But as $\SP\SL^0$ is germwise coinitial in $\SM$, \autoref{prop: SPSL^0 coinitial in wt(SM)}, there is $\un{\g}\in{}_\d\SP$ such that $[\wt{m}]>[\pzq_{\un{\g}}]$; together these statements imply that $\g>\g_0$ implies that $[\pzq_\g]>[\pzq_{\un{\g}}]$.
\end{proof}
    Using the previous proposition as motivation, we have the following definition.
\begin{definition}\label{def: asymp tot ord set of germs}
    Suppose, for some $\d\in\mu(0)_+$, that we have a coinitial subset $\La\subset\SN_\d$ with the induced total order. Then we say that a set of germs $\SH\subset\SM^0$ indexed by $\La$, ie., $\boldsymbol{\SH}=\{[m_\la]:\la\in\La\}$, \textbf{is asymptotically totally ordered by} $\boldsymbol{\La}$ if for each $\a_0\in\La$, there is a smaller $\b_0\in\La$ with the property that if $\b\in\La$ with  $\b_0>\b$, then $[m_{\a_0}]>[m_\b]$.
    We call $\boldsymbol{(\SH,\La)}$ \textbf{a germ scale at} $\mathbf{[0]}$
\end{definition}
    The previous proposition says that $\pzQ$ is asymptotically totally ordered by ${}_\d\SP$. One can show that are numerous such subsets of $\SM^0$ some more rigid than the above convenient subset $\pzQ$.
    We will now show that any Hausdorff topology on $\boldsymbol{\SG}=\SG_{1,1}$ that has reasonable properties with respect to germ order cannot be first countable. For this construction, we have $\wt{\SM}\subset\SG_{1,1}$ by extending $[m]\in\wt{\SM}$ to negative values by defining $m(x)=m(|x|)$ for sufficiently small $x\in\bbr$.
\begin{definition}\label{def: pzT preserves germ scales}
    We say  that a  topology $\pmb{\pzT}$  on $\SG$ is \textbf{asymptotically compatible with germ order} if the following holds. There is $\FX\subset\SG$ that is germwise coinitial in $\wt{\SM}$ such that $[0]\in cl_{\pzT}(\FX)\ssm\FX$ with the following property. For every $[m_0]\in\FX$, there is $U_0$ in $\pzT$ neighborhoods of $[0]$ such that if $[m]\in\FX$ and $[m]>[m_0]$, then $[m]\not\in U_0$.
\end{definition}
\begin{proposition}\label{prop: tau preserves germ scales}
    $\tau$ is asymptotically compatible with germ order.
\end{proposition}
\begin{proof}
    For each $\g\in{}_\d\SP$, let $[\pzp_\g]\in\SP\SL^0$ with $\rz\pzp_\g(\d)=\g$ and let $\FP=\{[\pzp_\g]:\g\in{}_\d\SP\}$. By definition, ${}_\d\FB$ is coinitial in $\SN_\d$ so that $[0]\in cl_\tau(\FB)$. Furthermore, if $\g_0\in{}_\d\SP$, then, essentially by definition, $[\pzp_\g]\not\in U^\d_{\g_0}$ for $\g>\g_0$.
    But if $[\pzp_\g]\in\FP$ with $[\pzp_\g]>[\pzp_{\g_0}]$, then certainly $\rz\pzp_\g(\d)>\rz\pzp_{\g_0}(\d)$, eg., $[\pzp_\g]\not\in U_{\g_0}$.

\end{proof}

\begin{remark}
   One can  strictly strengthen the notion that $\tau$ preserves germ order by, eg., using the asymptotic total order on the nice germwise coinitial subset $\pzQ$.  For example, using the asymptotic total order, on can demand that for a given neighborhood $U$ of $[0]$ in $\pzT$,  there is a $\la_0$ such that if $\la<\la_0$, then $[m_\la]\in U$.  We still get that $\tau$ satisfies these stronger criteria. Nonetheless, as the following theorem shows, the given conditions are sufficient to prevent first countability of a hypothetical topology $\pzT$ satisfying them. We first need a lemma.
\end{remark}
\begin{lemma}
   Suppose that $\FX$ is a germwise coinitial subset of $\wt{\SM}$ and  that $\SC\subset\FX$ is a countable subset. Then there is $[\un{m}]\in\FX$ with $[m]>[\un{m}]$ for all $[m]\in\SC$.
\end{lemma}
\begin{proof}
   Let $\d\in\mu(0)_+$. Then ${}_\d\SC\subset\SN_\d$ is countable so that there is $\un{\Fr}\in\SN_\d$ with $\SC\subset\wt{\SM}_{\un{\Fr}}$. But then \autoref{prop: ptwise bndd subset unif bndd} implies that there is $[m_0]\in\SM$ with the property that $[m]\in\SC$ implies $[m]>[m_0]$ and the result follows since, by hypothesis, there is $\un{m}\in\FX$ with $[m_0]>[\un{m}]$.
\end{proof}
\begin{theorem}\label{thm: pzT has germ scales->not 1st countable}
    Suppose that a topology $\pzT$ on $\SG$ is asymptotically compatible with germ order. Then $\pzT$ is not first countable.
\end{theorem}
\begin{proof}
    If a first countable topology exists, its restriction to $\wt{\SM}\subset\SG$ will be a first countable topology. Therefore, we will verify that such does not exist on $\wt{\SM}$.
    Suppose, by way of contradiction, that $[0]$ has a countable neighborhood base $U_1\supset U_2\supset\cdots$ at $[0]$ and let $\FX\subset\SG$ be as in \autoref{def: pzT preserves germ scales} so that there is $[m_j]\in\FX\cap U_j$ for all sufficiently large $j$ (so that we may assume that the index starts at 1). This implies that if $\SW=\{[m_j]:j\in\bbn\}$, then $[0]\in cl_\pzT(\SW)$. On the other hand, by the previous lemma, there is $[\un{m}]\in\FX$ with $[m_j]>[\un{m}]$ for all $j$. Therefore, by hypothesis, there is a neighborhood $\un{U}$ of $[0]$ in $\tau$ with the property that $[m_j]\not\in \un{U}$ for all $j$ (as $[m_j]>[\un{m}]$ for all $j$). That is, $\SW\cap\un{U}$ is empty, contradicting that $[0]$ is in the closure of $\SW$.
\end{proof}
   {\it Note that we already know that $\tau$ is not absorbing and therefore, although absolutely convex, is not locally convex. This result says that any topology $\pzT$ that is asymptotically compatible with germ order cannot be metrizable.}

\section{Continuity of composition and ring operations}
        In this section, armed with all of the preliminaries, eg., with the relationships among all of the various families of infinitesimals functioning as moduli, we can prove that operations on the spaces of continuous germs have good topological properties. That is, in the next subsection, we verify good topological properties of the ring properties and, in the following, good compositional properties. We believe that this is the best that can be hoped for. From a positive perspective, with respect to the ring operations, we will prove that multiplication by elements of $\SM$ (as defined in \autoref{def:  SM and SM^0} and extended to elements of $\SG_0$) is, in fact, an open map. From a negative perspective, we give a concrete example, see the \hyperlink{example}{Example} in the next subsection, that if we compose with a germ that is only continuous at $0$, the ensuing map is distinctively noncontinuous.

\subsection{Topological properties of the ring operations of germs}\label{subsec: top properties ring structure}
    Although scalar multiplication is not continuous, in this subsection we will verify that the ring operations, product and addition of germs, is continuous in the topology $\tau$. We also prove that  multiplication by nice monotone germs is an open map, see  \autoref{prop: mult SG by elt SM is open}.
\subsubsection{}
    Here we will give straightforward proofs of the $\tau$ continuity of the germ product.

\begin{proposition}\label{prop: SG_0 is Hausdorff top ring}
    $(\SG,\tau)$ is a Hausdorff topological ring.
\end{proposition}
\begin{proof}
    We need to show that the ring operations are continuous and as the vector space addition is continuous by the definition of the topology, we need only prove the product is continuous. First, left and right multiplication by a given element of $\SG$ is continuous. First, multiplication by a germ $[f]\in\SG$ is continuous; for if $\Fs\in\SN_\d$ is $\norm{f}_\d$ and given $\Fr\in\SN$, then there is $\Ft\in\SN_\d$ such that $\Fs\Ft<\Fr$, and so multiplication  of the elements of $U_\Ft$ by $[f]$ ie., $[\x f](U_\Ft)$ is contained in $U_\Fr$. In fact, this shows that given $\Ft\in\SN_\d$, then certainly there are $\Fr,\Fs\in\SN_\d$ with $\Fr\Fs<\Ft$ and so $U_\Fr\cdot U_\Fs\subset U_\Ft$ as $\|fg\|_\d\leq\|f\|_\d\|g\|_\d$ (*'s are left out here).
    This shows that $([f],[g])\mapsto [f][g]$ is continuous in the following special case: $([f_d]:d\in D)$ and $([g_d]:d\in D)$ are nets converging to $[0]$ in $\tau^\d$, then $d\mapsto [f_d][g_d]$ converges to $[0]$ in $\tau^\d$. Given this let's verify that if we have nets $[f_d]\ra [f]$ (in $\tau^\d$) and $[g_d]\ra [g]$ (also in $\tau^\d$), then $[f_d][g_d]\ra[f][g]$ (in $\tau^\d$ also). First, $\tau^\d$ continuity of left and right multiplication at $[0]$ implies that $[f]([g_d]-[g])\ra [0]$ in $\tau^\d$ and $([f_d]-[f])[g]\ra [0]$ in $\tau^\d$. But the previous assertion says also that $([f_d]-[f])([g_d]-[g])\ra [0]$ in $\tau^\d$. Adding the previous three expressions (noting that addition is a continuous operation in $\SG$) gets $[f_d][g_d]\ra[f][g]$ (in $\tau^\d$) as we wanted.
\end{proof}
\begin{proposition}
    $\SG^0$ is a closed subring of $\SG$.
\end{proposition}
\begin{proof}
    We just need to prove that $\SG^0$ is a closed subspace of $\SG$. But this is the import of \autoref{thm:  convergnet of cont germs in SG converges to cont}.
\end{proof}

    As a matter of formality, we will introduce the routine extension of $\tau$ to germs with range in $(\bbr^p,0)$

\begin{definition}\label{def: SG{n,p}, SG^0{n,p}}
    For $n,p\in\bbn$, let $\SG_{n,p}$ denote the germs of maps at $0$ of maps $f:(\bbr^n,0)\ra (\bbr^p,0)$, and  $\SG^0_{n,p}$ these map germs that are germs of continuous maps.
    Considering $\SG_{n,p}$ as a Cartesian product of $p$ copies of the topological ring $\SG_{n,1}$, we give it the natural product topology. We give the subset   $\SG^0_{n,p}$, etc., the subspace topology.
\end{definition}
    It is clear from the definition above that the product $\tau$ topology defined on $\SG^0_{n,p}$ is generated by translates of open neighborhoods of the zero germ and that the system of open neighborhoods of the zero germ is generated by Cartesian products of the form $U^\d_{\vec{\Fr}}=U^\d_{\Fr_1}\x\cdots\x U^\d_{\Fr_n}$ where $\vec{\Fr}$ denotes the ordered $n$-tuple $(\Fr_1,\ldots,\Fr_n)\in\SN_\d^n$.
    Of course,   given such $\vec{\Fr}$, if $\un{\Fr}<\Fr_j$ for each $j$, then $U_{\un{\Fr}}\x\cdots\x U_{\un{\Fr}}\subset U_{\vec{\Fr}}$, eg., when checking convergence we may test with these diagonal neighborhoods.

\begin{proposition}
    $\SG^0_{n,p}$ is a Hausdorff\; $\SG^0_{n,1}$-module.
\end{proposition}
\begin{proof}
    This is clear: as a finite product of Hausdorff spaces (with the product topology), it is clearly Hausdorff. Also the continuity of the module operation $\SG_{n,1}\x\SG_{n,p}\ra\SG_{n,p}$ is clear as it's just the ring operation $\SG_{n,1}\x\SG_{n,1}\ra\SG_{n,1}$ on each of the $p$ coordinates.
\end{proof}
\subsubsection{Germ product is an open map}

    Quite analogously with Krull type topologies on filtered rings, multiplication by a nice germ is actually an open map, eg., a homeomorphism onto its image.

    First of all, we will verify that, for $[m]\in\SM$,  multiplication $[\x m]:\SM\ra\SM$ in the topology $\tau|\SM$ is an open map onto its image. We will then see how this works for $\SG=\SG_{1,1}$.
    We need some preliminaries. Recall that, for $\Fr\in\SN_\d$, our $\tau^\d$ neighborhood of $[0]$ are given by $U_\Fr=\{[f]\in\SG:\|\rz f\|_\d<\Fr\}$ and note that restricted to $\SM$, this neighborhood becomes $\{[m]\in\SM:\rz m(\d)<\Fr\}$. For the moment, for $[\wt{m}]\in\SM$, and recall (\autoref{def: SU([m]), SM([m]), M^Fr, etc}) also the sets $\SM([\wt{m}])_u=\{[m]\in\SM:\rz m(\e)<\rz\wt{m}(\e)\;\text{for}\;\e\in\mu(0)_+\}$. It's clear that if $\rz\wt{m}(\d)=\wt{\Fr}$, then $\SU([\wt{m}])\subset U_{\wt{\Fr}}$. We need to prove an approximate reverse inclusion.

    Note for the following statement, the neighborhoods (of $[0]$) $U_\Fr$ in our topology when restricted to $\wt{\SM}$ (extended to negative values as we have before) correspond to the sets $\SM^\Fr$.
\begin{proposition}\label{prop: mult SM by elt SM is open}
    If $[m]\in\SM$, then $[\x m]:\SM\ra\SM$ is an open map.
\end{proposition}
\begin{proof}
    It suffices to prove that the image of an open neighborhood of $0$ under $[\x m]$ contains an open neighborhood of $0$. So it suffices to prove that given $\Fr\in\SN_\d$, there is $\wh{\Fr}\in\SN_\d$ with $[\x m](U_\Fr)\supset U_{\wh{\Fr}}$. If $m_\Fr\in\SM$ is such that $\rz m_\Fr(\d)=\Fr$, then obviously $\SU([m_\Fr])\subset U_\Fr$, and so if $[\un{m}]\in\SM$, $[\x\un{m}](\SU([m_\Fr]))\subset [\x\un{m}](U_\Fr)$ and note that $[\x\un{m}](\SU([m_\Fr]))=\SU([\un{m}m_\Fr])$.
    But letting $[\wt{m}]$ in \autoref{prop: un(SM)^r subset un(SM)(m)} be equal to $[\un{m}m_\Fr]$, we know that there is $\wh{\Fr}\in\SN_\d$ with $U_{\wh{\Fr}}\subset \SU([\un{m}m_\Fr])$, finishing the proof.
\end{proof}
    In a similar way, we will verify that multiplication by elements of $\SM$ is an open map $\SG\ra\SG$.
    At this point we will be using $\boldsymbol{{}_m\wt{\SG}}={}_m\wt{\SG}_{n,1}$, recalling that we have the  notation ${}_m\wt{\SG}$ for the elements $[m]\in\SG$ with $|x|\leq|y|$ implies that $|g(x)|\leq|g(y)|$ and also that for any $\Fr\in\SN_\d$, we have $U_\Fr\cap\SG=U_\Fr\cap {}_m\wt{\SG}$, eg., that it suffices to work with elements of ${}_m\wt{\SG}$.
    For an element $[m]\in\wt{\SM}$, we can define for  $[g]\in{}_m\wt{\SG}_0$, the product $[mg]$ defined to be the germ of $x\mapsto m(|x|)g(x)$, so that we have a well defined product map $\wt{\SM}\x{}_m\wt{\SG}_0\ra{}_m\wt{\SG}_0$. Given this, we have the following generalization of the above result.
\begin{proposition}\label{prop: mult SG by elt SM is open}
    If $[m]\in\wt{\SM}$, then $[\x m]:{}_m\wt{\SG}\ra{}_m\wt{\SG}$ is an open map.
\end{proposition}
\begin{proof}
    This follows immediately from
    We need to show that if $[\wt{m}]\in\wt{\SM}$ and $\Fr\in\SN_\d$, there is $\wt{\Fr}\in\wt{\SN}_\d$ such that $[\x \wt{m}]U_\Fr\supset U_{\wt{\Fr}}$.
    First of all, it's clear that there is $[\un{m}]\in\wt{\SM}$ with $\SU([\un{m}])\subset U_\Fr$, ie., $\Fr=\rz\un{m}(\d)$ for some $[m]\in\wt{\SM}$. It's also clear that $[\wt{m}]\SU([\un{m}])=\SU([\wt{m}\un{m}])$. But by \autoref{cor: given m, U_r subset SU(m) some r}, there is $\wh{\Fr}\in\wt{\SN}_\d$ such that $U_{\wh{\Fr}}\subset\SU([\wt{m}\un{m}])$. Putting these inclusions together, we have
\begin{align}
    U_{\wh{m}}\subset\SU([\wt{m}\un{m}])=[\wt{m}]\SU([\un{m}])\subset [\wt{m}]U_\Fr.
\end{align}
\end{proof}

\subsection{Topological properties of germ composition}\label{subsec: top properties of germ composition}
    We return to the context of  \autoref{subsec: top properties ring structure} and consider the morphisms of these topological rings induced by continuous map germs.

    The composition of germs of maps is well known and routine and we will assume the general definitions and facts known. Heuristically, composition of functions obviously distorts domains and range and so it will be here, but for germs in $[f]\in\SG^0$, composition works sufficiently well as $\rz f(\mu(0))\subset\mu(0)$.
     More specifically, if $[h]\in\SG_{n,n}$ or even in $\SG^0_{n,n}$, and $[f]\in\SG$, then $\rz f|_{B_\d}\circ \rz h|_{B_\d}$ is often not defined. On the other hand, if $[h]\in\SG^0_{n,n}$ and $[f]\in\SG_n$, then $\rz(f\circ h)_{B_\d}$ is always defined even when $\rz h(B_\d)\nsubseteq B_\d$ as $h(\mu(0))\subset\mu(0)$
     so that $\rz f|\rz h(\mu(0))$ is defined,
     and although we are defining the composition with $[f]$ in terms of its representative, any such is uniquely defined on all of $\mu(0)$. Therefore, the following definitions are well defined.
\begin{definition}\label{def: rc_[h],lc_[h]}
    If $[h]\in\SG^0_{n,n}$, we will denote the map $\SG_{n,p}\ra\SG_{n,p}:[f]\mapsto [f]\circ[h]\;\dot=[f\circ h]$ by $\mathbf{rc_{[h]}}$ and if $[g]\in\SG_{p,p}$, we define $\mathbf{lc_{[g]}}: \SG_{n,p}^0\ra\SG_{n,p}$ by $[f]\mapsto[g]\circ[f]\;\dot=[g\circ f]$.
\end{definition}
    To begin with, we look at the effect of composition on our sets of moduli. As right composition carries algebraic operations, we start there and as moduli are determined in terms of one dimensional mappings we will consider both the right and left actions of $\SM^0$ on inself.
    We will begin with some preliminaries on the effects of compositions on our semirings of moduli.

\begin{lemma}\label{lem: m(N_d)-N_d and N_(m(d))=N_d}
    Suppose that $\d\in\mu(0)_+$ and $\Fr\in\SN_\d$. Then $\SN_\Fr$ is coinitial with $\SN_{\d}$.
\end{lemma}
\begin{proof}
    First suppose that $\Fr=\rz m(\d)$ for some $[m]\in\SM^0$.
    Then $[\un{m}^{-1}]\in\SM$ and so any $[m]\in\SM$ can be written as $[\un{m}^{-1}][\wt{m}]$ for some $[\wt{m}]\in\SM$, ie., $[\un{m}^{-1}]\circ\SM=\SM$. Therefore, using  \autoref{lem: SN_d=SM_d}, we have
  \begin{align}
    \SN_\d=\{\rz m(\d):[m]\in\SM\}=\{\rz\un{m}\circ\un{m}^{-1}\circ m(\d):m\in\SM\}\qquad\qquad\qquad\notag\\
    =\{\rz m\circ\wt{m}(\d):[\wt{m}]\in[\un{m}^{-1}]\circ\SM\}\quad\qquad\qquad\notag\\
    =\{\rz \un{m}\circ \wt{m}(\d):[\wt{m}]\in\SM\}=  \SN_{\rz m(\d)}.
  \end{align}
    Next, suppose that $\d$ is serial (see \autoref{def: serial point}) so that $\d=\rz d_\om$ for some sequence $d_1>d_2>\cdots$ tending to $0$ in $\bbr_+$ and $\Fr=\rz m(\d)$ for an arbitrary $[m]\in\SM$. Then, as $m(d_j)>m(d_{j+1})$ for all sufficiently large $j$, there is $[m_0]\in\SM^0$ with $m_0(d_j)=m(d_j)$ for all sufficiently large $j$, so that $\Fr=\rz m_0(\d)$ and so by the first part of the proof, $\SN_{*m(\d)}=\SN_{*m_0(\d)}=\SN_\d$. That is, if $[m]\in\SM$, then for $\d$ serial, we have ${}_\d(\SM\circ[m])={}_\d\SM$, eg., they are coinitial with each other. But then by \autoref{cor: B_d1 coinit<->B_d2 coinit}, if $\wt{\d}\in\mu(0)_+$ is any infinitesimal, then ${}_{\wt{\d}}(\SM\circ[m])$ is coinitial with ${}_{\wt{\d}}\SM=\SN_{\wt{\d}}$ and finally note that as $[m]$ varies in $\SM$, ${}_{\wt{\d}}(\SM\circ[m])$ varies over the sets $\SN_{\wt{\Fr}}$ for $\Fr\in\SN_{\wt{\d}}$, ie., for $\wt{\Fr}\in\SN_{\wt{\d}}$, $\SN_{\wt{\Fr}}$ is coinitial with $\SN_{\wt{\d}}$.

\end{proof}


\begin{proposition}\label{prop: rc_h:G^0_n->G^0_n is C^0}
    If $[h]\in\SG^0_{n,n}$, then $rc_{[h]}:\SG^0_n\ra\SG_{n}^0$ is a continuous homomorphism.
\end{proposition}
\begin{proof}
    As right composition by $[h]\in\SG_{n,n}$ is an additive homomorphism $rc_{[h]}:\SG_n\ra\SG_n$ and translation by an element of $\SG^0_{n,n}$ is a homeomorphism, it suffices to show the following. If $([f_d:d\in D])$ is a net in $\SG^0_n$ converging to the zero germ $[0]$ in say the $\tau_\d$ topology, it follows that $([f_d\circ h]:d\in D)$ converges also (in some $\tau_\Fs$ topology for some $0<\Fs\sim0$, as topology is independent of $\Fs$).
    Given this, if $0<\d\sim 0$, $[f]\in\SG^0_n$ and  $[h]\in\SG^0_{n,n}$ and $\rz h(B_\d)\subset B_\e$ for some positive $\e\sim 0$, then $|\rz f\circ h(\xi)|\leq \norm{\rz f}_\e$ for $\xi\in B_\d$, so that if $\Fr=\norm{\rz h}_\d$, then $\xi\in B_\d$ implies that $|\rz f\circ h(\xi)|\leq\norm{\rz f}_\Fr$ and so $\norm{\rz f\circ h}_\d\leq\norm{\rz f}_\Fr$.
   The above estimate gives $\norm{\rz f_d\circ h}_\d\leq\norm{\rz f_d}_\Fr$ for all $d\in D$. Now by \autoref{lem: m(N_d)-N_d and N_(m(d))=N_d},  $\SN^0_\Fr$ is convergently coinitial in $\SN^0_\d$ (in fact equal) and $[f_d]$ converges in the $\tau_\Fr$ topology and so in the $\tau_\d$ topology and so $\{\norm{\rz f_d}_\Fr:d\in D\}$ is convergently coinitial in $\SN^0_\d$. But the estimates above then imply that $\{\norm{\rz f_d\circ h}_\d:d\in D\}$ is convergently coinitial in $\SN^0_\d$, as we wanted.
\end{proof}
\begin{corollary}\label{cor: rc_h:G^0_n,p,0->G^0_n,p,0 is C^0}
    Suppose that $[h]\in\SG^0_{n,n}$. Then $rc_{[h]}$ is a continuous $\SG^0_n$ module homomorphism of $\SG_{n,p}$.
\end{corollary}
\begin{proof}
    This is clear from the previous proposition.
\end{proof}
    Before we proceed to proving that left composition is a continuous operation, we want a more explicit description of the convergence of a net $([f_d]:d\in D)\subset\SG_{n,p}$ to $[f]\in\SG^0_{n,p}$. This result once more indicates the uniform convergence flavor of $\tau$ convergence.
\begin{lemma}\label{lem: ptwise condition for [f_d]->[f_0] in tau}
    Suppose that $([f_d]:d\in D)$ is a net in $\SG^0_{n,p}$ and $[f]\in\SG^0_{n,p}$. Then $[f_d]\ra[f]$ in $\tau$ if and only if the following holds. Let $\d\in\mu(0)_+$. Given $\Fr$ in $\SN_\d$, then there is $d_0\in D$ such that if $\xi\in\mu_n(0)$ satisfies $\xi\in B_\d$, then $\rz f_d(\xi)\in\mu_p(0)$  satisfies $|\rz f_d(\xi)-\rz f(\xi)|\leq\Fr$ for all $d>d_0$, ie., if $\|\rz f_d-\rz f\|_\d\leq\Fr$.
\end{lemma}
\begin{proof}
     This is a direct consequence of the definition.
\end{proof}
     Left composition by an element of $\SG^0_{n,n}$ acting on $\SG^0_{n,p}$ is obviously not a homomorphism, but we have a good topological result. Note also that, unlike proving the $\tau$ continuity of right composition, proving the continuity of left composition does not follow immediately from such at $[0]$ upon translation.
     The proof of left continuity will be a consequence of  a second criterion for a germ in $\SG$ to be continuous.
     This second criterion for the continuity of a germ is closely related to the first given in \autoref{prop: germ continuity from k<<<d}.
\begin{proposition}\label{prop: 2nd criterion for germ in C^0}
     Given $[f]\in\SG$, we have $[f]\in\SG^0$ if and only if the following criterion holds. Fix $\d\in\mu(0)_+$. Then, for each $\Fr\in\SN_\d$, there is $\Fs\in\SN_\d$ such that if $\xi,\z\in B_\d$ satisfy $|\xi-\z|<\Fs$, then $|\rz f(\xi)-\rz f(\z)|<\Fr$.
\end{proposition}
\begin{proof}
     Recalling the criterion in \autoref{prop: germ continuity from k<<<d} for a germ $[f]\in\SG$ to be continuous (see \autoref{def: e is [f]-good}), it's clear that, in light of the equivalence in \autoref{lem: la<Fr all Fr in Nd->la<<<Fr,also}, that the above criterion for $[f]\in\SG$ holds if $[f]\in\SG^0$. Conversely, we will prove that the above criterion implies that $[f]$ has strong $\d$-good numbers and then invoke \autoref{prop: germ continuity from k<<<d}. To this end, notice that the above criterion implies that for each $\Fr\in\SN_\d$, there is $\Fs\in\SN_\d$ such that
   \begin{align}
       \rz S([f],\Fs)\;\dot=\;\rz\sup\{|\rz f(\xi)-\rz f(\z)|:\xi,\z\in B_\d,\;\text{and}\;|\xi-\z|<\Fs\}\leq\Fr.
   \end{align}
    Given this, let $\ov{\k}\in\mu(0)_+$ with $\ov{\k}\lll\d$ and let
   \begin{align}
       \D(\d,\ov{\k})=\{(\xi,\z)\in B_\d\x B_\d:|\xi-\z|<\ov{\k}\}.
   \end{align}
      Suppose that $(\xi,\z)\in\D(\d,\ov{\k})$ and that $\Fr\in\SN_\d$ is arbitrary, then, by hypothesis, there is $\Fs\in\SN_\d$ such that  $\rz S([f],\Fs)\leq\Fr$ holds. But $\ov{\k}<\Fs$ and so $\rz S([f],\ov{\k})<\Fr$ holds and as $\Fr$ was chosen arbitrarily in $\SN_\d$, \autoref{lem: la<Fr all Fr in Nd->la<<<Fr,also} implies that $\rz S([f],\ov{\k})\lll\d$, or equivalently, that $\rz S([f],\ov{\k})<\k$ for some $\k\lll\d$, ie., $\k$ is, by definition, strongly $[f]$-good for $\d$, as we wanted to show.
\end{proof}
    We say that, $[S]$, the germ at $0$ of a subset $S\subset\bbr$ has nontrivial interior at $0$, if there is a *interval $\SI=(\a,\b)\subset\mu(0)_+\cap B_\d$, with $\b-\a\not\lll\d$ ,  such that $\SI\subset\rz S$. Giving this, consider the function $h:(\bbr,0)\ra (\bbr,0)$ given by $h(x)=x$ if $x\in\bbq$ and $h(x)=0$ if $x\in\bbr\ssm\bbq$ is continuous at $0$, but it's germ at $0$ is not a continuous germ. We then assert that if $[f]\in{}_m\SG_{1,1}$ is continuous on $[S]$ for some $[S]$ with nontrivial interior at $0$,  then $[h\circ f]\not\in\SG^0$. This can be seen as follows. It suffices to consider the case when $[f]=[m]$ for some $[m]\in\SM$ (with $[m]\in U_\Fr$ for some $\Fr\in\SN_\d$) extended by (anti)symmetry across $0$ such that, if $(\a,\b)\subset\rz S$ is our nontrivial interval, $\rz m|(\a,\b)$ satisfies the criterion for germ continuity in \autoref{prop: germ continuity from k<<<d} on $\SI\dot=(\a,\b)$. By *transfer of the consequences that $[m]$ restricted to intervals in $S$ is a local monotone homeomorphism, it follows that  $\rz m(\SI)$ is a *interval $\SJ\dot=(\g,\la)\subset B_\d$, and therefore that $\SX\dot=\rz\bbq\cap\SJ$ is *dense in $\SJ$ as is $\SY\dot=(\rz\bbr\ssm\rz\bbq)\cap\SJ$. From $[m]\in U_\Fr$ also follows that $\rz m(\Ft)\not\lll\Ft$ for $\Ft\in B_\d$ and therefore the condition $\b-\a\not\lll\d$ implies that if  $\Ft\in\SJ$ then $\Ft>\Fs$ for some $\Fs\in\SN_\d$.
    So, by definition of $[h]$, we have that  $\rz h\circ m|m^{-1}(\SX)=\rz h|\SX$ is zero, but for $\Ft\in m^{-1}(\SY)$, we have that $\rz h\circ m(\Ft)>\Fr$ for some $\Fr\in\SN_\d$.
    Given this, choosing $\la_1\in m^{-1}(\SX)$ and $\la_2\in m^{-1}(\SY)$  (so that eg., clearly $\la_2$ bigger than $\Fr$ for some $\Fr\in\SN_\d$) and with $|\la_1-\la_2|\lll\d$, we have that $|\rz h\circ m(\la_1)-\rz h\circ m(\la_2)|=\rz h\circ m(\la_2)>\Fr$.
    But then according to the criterion for continuity of a germ, \autoref{prop: germ continuity from k<<<d}, $[h\circ m]$ is not continuous.

    After the next proposition we will show  that composition, $lc_{[h]}$ is not continuous if $[h]$ is the above function. That is, we will show that there is  a net $([f_d]:d\in D)$ of germs in $\SG_{1,1}$ with $[f_d]\ra 0$ such that $[h\circ f_d]\not\ra 0$.
    In contrast, we are now in a position to prove the following proposition.
\begin{proposition}\label{prop: lc_h:G^0_n,p,0->G^0_n,p,0 is C^0}
    Suppose that $[h]\in\SG^0_{p,p}$, then $lc_{[h]}:\SG^0_{n,p}\ra\SG^0_{n,p}$ is a continuous map.
\end{proposition}
\begin{proof}
    Suppose that $([f_d]:d\in D)$ is a net in $\SG^0_{n,p}$ such that $[f_d]\ra [f]\in\SG^0_{n,p}$ in the topology $\tau$. We want to show that $[h]\circ [f_d]=[h\circ f_d]\ra [h\circ f]$ in $\tau$.
    By \autoref{lem: ptwise condition for [f_d]->[f_0] in tau}, it suffices to prove, for a given fixed $\d\in\mu(0)_+$, the following statement. Given $\Fr\in\SN_\d$, there is $d_0\in D$ such that for $d>d_0$ and each $\xi\in B_\d$, we have $|\rz h\circ f_d(\xi)-\rz h\circ f(\xi)|<\Fr$. Fix this $\Fr\in\SN_\d$ and notice that \autoref{prop: 2nd criterion for germ in C^0}  implies the following statement. $\pmb{(\diamondsuit)}$: Given the fixed $\Fr$, there is $\ov{\Fr}\in\SN_\d$ such that if $\xi\in B^n_\d$ and $\z\in B^p_\d$ satisfy $|\z-\rz f(\xi)|<\ov{\Fr}$ then we have that $|\rz h(\z)-\rz h(f(\xi))|<\Fr$.
    Then applying the hypothesis in the guise given by \autoref{lem: ptwise condition for [f_d]->[f_0] in tau} once more, we know that  there is $d_0\in D$, such that for $\xi\in B^n_\d$ and $d>d_0$, we have that $|\rz f_d(\xi)-\rz f(\xi)|<\ov{\Fr}$. This is precisely the condition, with $\z=\rz f_d(\xi)$ for $d>d_0$, required for the previous statement $(\diamondsuit)$ to hold.
\end{proof}
    Let's give an example to show how the hypothesis on $[h]$ in the previous proposition  cannot be weakened, ie., if $[h]$ is only assumed to be continuous at $0$ then $lc_{[h]}$ will not be continuous. We believe also that this example demonstrates some of the capacities of the tools developed in this paper.
\begin{example}\label{example: h not in SG^0-> lc[h] notC^0}
    We \hypertarget{example}{will assume},
     for construction's sake, that $\d$ is serial (see \autoref{def: serial point}). (By \autoref{cor: tau^d_1 = tau^d_2}, the choice of $\d$ is irrelevant as far as determining convergence of the nets in this example.) Using the notation in the next section, see \autoref{def: SSSQ, SSSQ_FU, [(r_i)]}, we have an element $[(d_j)]\in\SS\SQ$ with $\d=\rz d_\om$ and we have (see the next section) that $\SN_\d=\{\rz r_\om:[(r_j)]\in\SS\SQ\}$ and we may further assume that $\d\in\rz\bbq_+$ by assuming that $d_j\in\bbq$ for all sufficiently large $j\in\bbn$. If $Irr$ denotes $\bbr_+\ssm\bbq_+$, then given $\un{\Ft}\in\SN_\d$, we will find  $[(s_i)]\in\SS\SQ$ with $s_i\in Irr$ for all sufficiently large $j$ (so that $\rz s_\om\in\rz Irr$)  and also satisfying $\rz s_\om<\un{\Ft}$. We will therefore have a subset $\SS\subset\SN_\d$ with the two properties (a) $\SS\subset\rz Irr$ and (b) $\SS$ is coinitial in $\SN_\d$. Assuming this data for the moment, let $h:\bbr\ra\bbr_+$ be the function $h(x)=0$ if $x\in\bbq$ and $h(x)=|x|$ for $x\in Irr$ for all $x$ sufficiently small and extend the domain of definition of elements $[m]\in\SM$ to a neighborhood of $0$ symmetrically, ie. by $x\mapsto m(|x|)$ noting that, by definition of $\SM$, the extension is a continuous germ if and only if the original element of $\SM$ is. For our convergent net, choose our directed set to be $\SN_\d$ with the given (total) order reversed (to be consistent with the conventions in this paper), and for every $\Fr\in\SN_\d$, choose $[m_\Fr]\in\SM^0$ with $\rz m_\Fr(\d)=\Fr$. We have that the net $\Fr\mapsto[m_\Fr]$ indeed converges to $[0]$ by \autoref{prop: ptwise bndd subset unif bndd} and so if $[m]\in\SM^0$, then $\Fr\mapsto [\wt{m}_\Fr]\;\dot=[m+m_\Fr]$ converges to $[m]$ by the definition of the topology. In particular, choose $[m]$ to be the germ at $0$ of $x\mapsto |x|$ so that, eg., $\rz m(\d)\in\rz\bbq$ by the choice of $\d$. But note that for $\Fs\in\SS$, we have that $[\wt{m}_\Fs]$ satisfies $\rz\wt{m}_\Fs(\d)=\d+\Fs$, ie., we have (c) $\rz\wt{m}_\Fs(\d)\in\rz Irr$ and (d) $\rz\wt{m}_\Fs(\d)>\d$. We now have all components for our counterexample. By the definition of $[h],[m]$ and $\d$, we have that $\rz h\circ m(\d)=0$ but for $\Fs\in\SS$, (c), (d) and the definition of $[h]$ implies we have $\rz h\circ\wt{m}_\Fs(\d)>\d$, eg., for all $\Fs$ in the coinitial subset $\SS$ of $\SN_\d$, we have  $|\rz(h\circ\wt{m}_\Fs)(\d)-\rz(h\circ m)(\d)|>\d$, which, again by \autoref{prop: ptwise bndd subset unif bndd}, implies that $[h\circ m_\Fr]$ does not converge to $[h\circ m]$.
\end{example}
    It seems that as long as $[m]\in\SM$, and we can choose serial $\d$ with $\rz m(\d)\in\rz\bbq$, then defining $h(x)=m(|x|)$ for $x\in Irr$ and $=0$ for $x\in\bbq$, we can get the same nonconvergence result. In particular, note that this includes $[m]\in\SM^0$ that are arbitrarily flat! (See our conclusion for a further note.)
    {\it So there exists germs $[g]\in{}_m\SG$ arbitrarily $\tau$ close to a given continuous germ with the property that left composition with $[g]$ is not continuous.}

 We will now give an introduction to the properties of composition with germs of homeomorphisms, finishing with a proof that the group (under composition) of germs of homeomorphisms of $(\bbr^n,0)$ is a topological group.

\begin{definition}
    Recall that $[f]\in\SG^0_{n,n}$ is a homeomorphism germ if there exists $[g]\in\SG^0_{n,n}$, its inverse, satisfying $[f]\circ[g]=[id]=[g]\circ [f]$. This set of germs is clearly a group, which we will denote by $\boldsymbol{\SH^0_n}$.
\end{definition}
    From the previous work, we have the following.
\begin{corollary}
    If $[f]\in\SH^0_n$ and $[g]\in\SH^0_p$, then
\begin{enumerate}
  \item  $lc_{[g]}:\SG^0_{n,p}\ra\SG^0_{n,p}$ is a homeomorphism.
  \item  $rc_{[f]}:\SG^0_{n,p}\ra\SG^0_{n,p}$ is a $\SG^0_n$-module isomorphism.
\end{enumerate}
\end{corollary}
\begin{proof}
 The first assertion is a consequence of \autoref{prop: lc_h:G^0_n,p,0->G^0_n,p,0 is C^0} and the second a consequence of \autoref{cor: rc_h:G^0_n,p,0->G^0_n,p,0 is C^0} once one makes the routine observation
     As $\a=[h]$ and $\b=[h]^{-1}$ are continuous maps with $\a\circ\b$ and $\b\circ\a$ the identity map on $\SG^0_{n,p}$, the result follows in the usual way.
\end{proof}


   It's clear that the next step in this investigation is to see if with the topology $\tau$, $\SH^0_n$ is a topological group, ie., check the continuity of the maps $\SH^0_n\x\SH^0_n\ra\SH^0_n$ given by $([f],[g])\ra[f\circ g]$ and $\SH^0_n\ra\SH^0_n$ given by $[f]\ra[f^{-1}]$ with respect to the topology $\tau^\d$.
   Before we verify this, we need some preliminaries.

   Below we will use the following formulation of the product topology on $\SG^0_{n,n}$. Using standard canonical coordinates on $\bbr^n$, write $f:(\bbr^n,0)\ra(\bbr^n,0)$ as $(f^1,\ldots,f^n)$ and so with germs. Given this one can verify that the topology defined by the product topology $\tau^\d\x\cdots\x\tau^\d$ on $\SG^0_{n,n}$ is equivalent to that defined by $\rz\|f\|_\d=\rz\|(f^1,\cdots,\;f^n)\|_\d=\sum_j\rz\|f^j\|_\d$. That is, the system of neighborhoods of $[0]$ in $\SG^0_{n,n}$ given by
  \begin{align}
      {}_nU_\Fr([0])=\{[f]\in\SG^0_{n,n}:\sum_j\rz\|f^j\|_\d<\Fr\}
  \end{align}
    as $\Fr$  varies in $\SN_\d$ gives a subbase for this product topology on $\SG^0_{n,n}$. So, for notational simplicity, \textbf{below we will use} $\boldsymbol{\rz\|f\|_\d}$ \textbf{to denote} $\boldsymbol{\mathbf{\sum_j}\rz\|f^j\|_\d}$ \textbf{if} $[f]\in\SG^0_{n,n}$.
    Note that the triangle inequality holds for this `norm', ie., if $[f],[g]\in\SG^0_{n,n}$, then $\rz\|f+g\|_\d\leq\rz\|f\|_\d+\rz\|g\|_\d$.
\begin{lemma}\label{lem: h(B_d) contains B_d/2}
     Let $[h]\in\SH^0_n$.
     First, if $\d\in\mu(0)_+$, then there is $\Fr\in\SN_\d$ such that $B_\Fr\subset\rz h(B_\d)$. Second, there is a $\tau$ neighborhood $U$ of $[h]$ in $\SH^0_n$ and $\Fs\in\SN_\d$ such that $[g]\in U$ implies that $B_\Fs\subset\rz g(B_\d)$.
\end{lemma}
\begin{proof}
     Suppose the first statement is false, so that $\rz h(B_\d)\nsupseteqq B_\Fr$ for all $\Fr\in\SN_\d$. Then, there is $\xi\in B_\d$ with $|\xi|\in\SN_\d$ and $\z=\rz h(\xi)$ satisfying $|\z|=\k\lll\d$, ie., $|\rz h^{-1}(\z)|\ggg|\z|$. Note then that the germ $[m]$ in $\SM^0$ with representative given by the map $t\mapsto\sup\{|h^{-1}(x)|:|x|=t\}$ satisfies $\rz m(\k)\ggg\d$, impossible for an element of $\SM$.

     For the second assertion, let's first prove the assertion for $[h]$ in some neighborhood of the identity map $id$. By \autoref{cor: all standard descriptions tau}, we know that if $[m]\in\SM^0$ and  $U=\{[h]\in\SH^0_n:\rz\|h-id\|_\Ft<\rz m(\Ft),\;\text{for}\;\Ft\in\mu(0)_+\}$, then $U$ is a $\tau$ neighborhood of $[h]$. It suffices to choose our neighborhood $U$ with $[m]$  decaying weakly at $0$: $m(t)=t^2$. So as $\Ft$ increases in $\mu(0)_+$, $\Ft\mapsto \Ft-\rz m(t)$ is increasing and eg., for $\d\leq\Ft\in\mu(0)_+$  $\Ft-\rz m(\Ft)\geq\d/2$.
     Given these preliminaries, $[h]\in U$ implies $\rz|h(\xi)|\geq|\xi|-\rz m(|\Ft|)$ for $|\xi|=\Ft\in\mu(0)_+$, eg., $|\xi|\geq\d_0$ implies $|h(\xi)|\geq\d_0/2$ by the previous sentence. Stated conversely, if $|h(\xi)|<\d_0/2$, then $|\xi|<\d_0$. But as  $\rz h$ is a bijection on $\mu(0)$, then this is the same as saying that if $|\z|<\d_0/2$, then $|h^{-1}(\z)|<\d_0$, ie., $\rz\|h^{-1}\|_{\d_0/2}\leq\d_0$.
     So as this is stated for any fixed $\d_0\in\mu(0)_+$, we now have the following statement: if $[h]\in U$, then for all $\d\in\mu(0)_+$, we have $\rz\|h^{-1}\|_\d<2\d$.
      Next, it's easy to check that, for a given $[h]\in\SH^0_n$ we have $\rz\|h^{-1}\|_\d=\rz\inf\{\Fr\in\rz\bbr_+:B_\d\subset h(B_\Fr)\}$.
      With this, our previous calculation implies if $[h]\in U$, then $B_{\d/2}\subset \rz h(B_\d)$.
      Finally, for general $[h_0]\in\SH^0_n$,  $rc_{[h_0]}:\SH^0_n\ra\SH^0_n$ is a homeomorphism implies $rc_{[h_0]}(U)=\{[h\circ h_0]:[h]\in U\}$ is a $\tau$ neighborhood of $[h_0]$ in $\SH^0_n$. Since we already have $\rz h^{-1}_0(B_\d)\supset B_\Fr$ for some $\Fr\in\SN_\d$ by the first assertion, then the result follows from the fact (just proved) that $[h]\in U$ implies that $\rz h(B_\Fr)\supset B_{\Fr/2}$.
\end{proof}
\begin{proposition}\label{prop: SH^0_n is top group}
   $(\SH^0_n,\tau)$ is a topological group.
\end{proposition}
\begin{proof}
   Let's first verify that the product map is continuous. This will be a consequence of the following. Suppose that $[f],[g],[h]\in\SH^0_n$ and that $(([f_d],[g_d]:d\in D))$ is an upward directed net in $\SH^0_n\x\SH^0_n$ such that $([f_d],[g_d])\ra([f],[g])$ in the product topology $\tau\x\tau$ on $\SH^0_n\x\SH^0_n$. We will verify that $[f_d]\circ[h]\circ[g_d]\ra[f]\circ[h]\circ[g]$ in $\tau$ by showing that for $\Ft\in\SN_\d$, there is $\wh{d}\in D$ such that $d>\wh{d}$ implies that $[f_d]\circ[h]\circ[g_d]\in U_\Ft([f]\circ[h]\circ[g])$.
   Fix this arbitrary $\Ft\in\SN_\d$ and note first that \autoref{prop: lc_h:G^0_n,p,0->G^0_n,p,0 is C^0} implies that there is $d_1\in D$ such that $d>d_1$ gives \textbf{(a)}: $\rz\|f_d\circ h\circ g-f\circ h\circ g\|_\d<\Ft/2$. Fixing $\wt{d}>d_1$, \autoref{cor: rc_h:G^0_n,p,0->G^0_n,p,0 is C^0} implies that there is $d_2\in D$ with \textbf{(b)}: $\rz\|f_{\wt{d}}\circ h\circ g_d-f_{\wt{d}}\circ h\circ g\|_\d<\Ft/2$. But then for $d$ in $D$ greater than $d_1$ and $d_2$, the *triangle inequality gets
  \begin{align}
     \rz\|f_d\circ h\circ g_d-f\circ h\circ g\|_\d\leq\rz\|f_d\circ h\circ g-f\circ h\circ g\|+\rz\|f_d\circ h\circ g_d-f_d\circ h\circ g\|_\d
  \end{align}
   with the first expression on the right hand side of the inequality less that $\Ft/2$ by fact (a) and the second less than $\Ft/2$ by fact (b).

   To prove the continuity of the map $[h]\mapsto[h^{-1}]$, let $([h_d]:d\in D)$ be an upward directed net with $[h_d]\ra[h]$ in $\tau$. That is, given $\Fs\in\SN_\d$, there is $d_\Fs\in D$ such that \textbf{(c)}: $\rz\sup\{|h_d(\xi)-h(\xi)|:\xi\in B_\d\}<\Fs$ for $d>d_\Fs$. Now for $d\in D$ sufficiently large, the fact that $[h_d]\ra[h]$ and \autoref{lem: h(B_d) contains B_d/2} implies there is $\Fr\in\SN_\d$ such that $B_\Fr\subset\rz h_d(B_\d)$. Therefore, for sufficiently large $\wt{d}\in D$, the expression (c) implies $\rz\sup\{|h_d(h_{\wt{d}}^{-1}(\z))-h(h_{\wt{d}}^{-1}(\z))|:\z\in B_\Fr\}<\Fs$ for $d>d_\Fs$. In particular, choosing $\wt{d}=d$ (sufficiently large), we have $\rz\sup\{|\z-h\circ h_d^{-1}(\z)|:\z\in B_\Fr\}<\Fs$ for $d>d_\Fr$; that is, $\{\rz\|h\circ h_d^{-1}-id\|_\Fr:d\in D\}$ is convergently coinitial with $\SN_\d$ and therefore with $\SN_\Fr$ (by \autoref{lem: m(N_d)-N_d and N_(m(d))=N_d}). So we have $[h\circ h_d^{-1}]\ra[id]$ in $\tau$ and as $[h^{-1}]\in\SG^0_{n,n}$, then \autoref{prop: lc_h:G^0_n,p,0->G^0_n,p,0 is C^0} implies that $[h^{-1}_d]=[h^{-1}]\circ [h\circ h_d^{-1}]\ra[h^{-1}]\circ[id]=[h^{-1}]$ in $\tau$.
\end{proof}
\section{Standard interpretations for $\tau$}\label{sec: standard interpret tau}
        In this  part, we will give  several  standard interpretations of our topology. Each depends on a particular choice of type of infinitesimal $\d$ and also on the kind of coinitial subset of $\wt{\SN}_\d$ that we associate with $\d$. In each case, we will give an explicit standard condition for net convergence in the topology $\tau$. The reader, upon seeing how the author is concocting these examples of standard subbases for $\tau$, will be able to use the various types of $\d\in\mu(0)_+$, coinitial subsets of $\wt{\SN}_\d$, and the fact that $\tau$ is invariant under such choices, to construct several other standard incarnations of $\tau$.
        We will first give interpretations in the case that our positive infinitesimal $\d$ is a serial point (see \autoref{def: serial point}) and for contrast then when $\d$ is a point 'far from serial'.

\subsection{Standard interpretation for $\d$ serial}\label{subsec: stand interpret for d serial}

\begin{definition}\label{def: serial point}
    By definition, a \textbf{serial point} in $\d$ in $\mu(0)_+$ (see eg.,  Puritz, \cite{Puritz1972}) is just $\rz d_\om$ for some choice of decreasing sequence $d_1,d_2,\ldots$ in $\bbr_+$ with limit $0$ and infinite index $\om\in\rz\bbn$.
\end{definition}
    The choice of $\om$ is equivalent to a choice of free ultrafilter $\FU$ on $\bbn$ (see below). It will follow that the first standard rendition of our topology will depend of such a pair $(\{d_j\}_{j\in\bbn},\FU)$. Nonetheless, from \autoref{subsec: tau independ of delta} we know that our topology is independent of $\d$ and, in particular,  when $\d$ is serial, independent of the choice of such pairs $(\{d_j\}_{j\in\bbn},\FU)$. We will examine the standard implications of this.

    Given this, we will begin with the set of moduli that most directly define $\tau$, ie., $\SN_\d$ where now $\d\in\mu(0)_+$ is now a serial point. So, as noted, there is a decreasing sequence $\pzD$ of positive real numbers $d_1,d_2,\ldots$ with  $\d\in\rz\pzD$; specifically, there is $\om\in\rz\bbn_\infty$ with $\d=\rz d_\om$. Note, then, that the free ultrafilter associated to the index $\om$ is just $\FU=fil(\om)=\{A\subset\bbn:\om\in\rz A\}$. When $\d$ is a serial point, we are able to dispense with our monotone family, $\wt{\SM}$ defining the moduli $\wt{\SN}_\d$ and replace it with a discrete analog of germs of sequences.
     This will allow us to give our weakest standard rendition for our topology $\tau=\tau^\d$.
\begin{definition}\label{def: SSSQ, SSSQ_FU, [(r_i)]}
    Let \boldmath$\SS\SQ$\unboldmath denote the set of all (germs of) decreasing sequences $(r_i)$ with $r_1,r_2,\ldots$ in $\bbr_+$ such that $\lim r_i=0$.  That is, $(r_i)\sim (s_i)$ if for some $j_0\in\bbn$, $r_i=s_i$ for $j\geq j_0$ gives an equivalence relation (germs at index infinity) and $\pmb{[(r_i)]}$ will denote the class containing $(r_i)$. Sometimes we will also denote this by $\pmb{[\vec{r}]}$. Clearly, this germ equivalence is also given by $(r_i)\sim (s_i)$ if and only if there is $t_0\in\bbr_+$ such that for $\max\{r_i,s_i\}<t_0$ we have $r_i=s_i$.
    We define the germ partial order on the germs in $\SS\SQ$ by $\pmb{[(r_i)]<[(s_i)]}$ if $r_i<s_i$ on a cofinite subset of $\bbn$.
    Let $\mathitbf{fuf}\pmb{(\bbn)}$ denote the set of free ultrafilters on $\bbn$ and if for $\FU\in fuf(\bbn)$, we let $\boldsymbol{\SS\SQ_\FU}$  denote the set of $\FU$ equivalence classes, $\pmb{[\vec{r}]}_{\boldsymbol{\FU}}$, of decreasing sequences. (This is just the $\SU$ ultrapower of decreasing sequences.) Clearly, there is a well defined (identification) map $\SS\SQ\ra\SS\SQ_\FU:[\vec{r}]\mapsto [\vec{r}]_\FU$ (as cofinite subsets of $\bbn$ belong to $\FU$).
    We have a second ``$\FU$-partial order'' denoted  $\pmb{[(r_i)]\stackrel{\FU}{<}[(s_i)]}$ on $\SS\SQ$ and simply denoted by $\pmb{[\vec{r}]}_{\boldsymbol{\FU}}\;\pmb{<}\;\pmb{[\vec{s}]}_{\boldsymbol{\FU}}$ on $\SS\SQ_\FU$ given, in both cases, by $(r_i)<(s_i)$ if $\{i:r_i<s_i\}\in\FU$.
\end{definition}


\begin{lemma}\label{lem: 3 items-rel between SQ and SN_d}
    With the above development, the following statements are easily checked.
  \begin{enumerate}
    \item $\SS\SQ=\{[(m(d_i))]:[m]\in\SM\}$ and so $\SN_\d=\{\rz r_\om:[(r_i)]\in\SS\SQ\}$.
    \item If $[(r_i)],[(s_i)]\in\SS\SQ$, then $[\vec{r}]<[\vec{s}]$ implies $[\vec{r}]\stackrel{\SU}{<}[\vec{s}]$ which implies $[\vec{r}]_\FU<[\vec{s}]_\FU$ which is equivalent to $\rz r_\om<\rz s_\om$.
    \item  If $\ST\subset\SS\SQ$, then $\{\rz r_\om:[(r_i)]\in\ST\}$ is coinitial in $\SN_\d$   if and only if for each $[(r_i)]\in\SS\SQ$, there is $[(s_i)]\in\ST$ such that $[\vec{s}]\stackrel{\FU}{<}[\vec{r}]$.
  \end{enumerate}
\end{lemma}


    Now for $j\in\bbn$, let $D_j=\{x\in\bbr:|x|\leq d_j\}$ and $F(j)$ the set of functions $f:(\bbr,0)\ra (\bbr,0)$ whose domain $dom(f)$ contains $D_j$. Define the typical supremum norm $\|\;\|_j$ on $F(j)$.
    Given a fixed $\FU\in fuf(\bbn)$ and $[(r_i)]\in\SS\SQ$, let
\begin{align}
    \pmb{\pzU}(\mathbf{r_i})=\{[f]\in\SG:\text{there is}\; A\in\FU\;\text{such that}\;\|f\|_i<r_i\;\text{for all}\;i\in A\}.
\end{align}
\begin{lemma}
     For $[(r_i)]\in\SS\SQ$, the maps $[(r_i)]\ra\pzU(r_i)$ and $[\vec{r}]_\FU\ra\pzU$ are well defined. If $[\vec{r}]\stackrel{\FU}{<}[\vec{s}]$, then $\pzU(r_i)\subset\pzU(s_i)$.
\end{lemma}
\begin{proof}
    The last statement is obvious; we will only verify that the assignment $[(r_i)]\mapsto\pzU(r_i)$ is well defined as the other proof is similar. Suppose that $[f]\in\pzU(r_i)$ and $[(s_i)]\in\SS\SQ$ satisfies $[(s_i)]=[(r_i)]$, we want to show that $\pzU(r_i)=\pzU(s_i)$. By symmetry, it suffices to verify that $\pzU(r_i)\subset\pzU(s_i)$. Now, by hypothesis, $C=\{i:s_i=r_i\}$ is cofinite in $\bbn$, eg., is in $\FU$ and $\{i:\|f\|_i<r_i\}=A\in\FU$ so that $A'=C\cap A\in\FU$. But, by the definition of $A'$, $A'\subset B\dot=\{i:\|f\|_i<s_i\}$ and so by filter properties $B\in\FU$, ie., $[f]\in\pzU(s_i)$.
\end{proof}
     When asserting the existence of representatives defined on a ball of a given size and having certain properties, we will assume the existence of these representatives and concentrate on the verification of the properties.
     Given these preliminaries, we have the following standard formulation of the $U_\Fr$'s, our subbase at $[0]$ for $\tau^\d$.
\begin{lemma}\label{lem: [f]in U_*r_om <->[f]in pzU(r_i)}
     Fix $[(d_i)]$ and $\om\in\rz\bbn$ so that $\d=\rz d_\om$ and $\FU=fil(\om)$.  Let $[f]\in\SG$, $[(r_i)]\in\SS\SQ$. Then $[f]\in U_{*r_\om}$ if and only if $[f]\in\pzU(r_i)$ .
\end{lemma}
\begin{proof}
     This is simple:  $[f]\in U^{\d_\om}_{*r_\om}$ means that  $\rz\|\rz f\|_{*d_\om}<\rz r_\om$, and this is equivalent to $\om\in\rz\{i:\|f\|_{d_i}<r_i\}$ which, in turn  means $\{i:\|f\|_{d_i}<r_i\}\in\FU$, ie., by definition $[f]\in\pzU(r_i)$.
\end{proof}

   Given this we have the following lemma.
\begin{lemma}\label{lem: stand descrip with FU}
    Suppose that $([f_\a]:\a\in \pzA)$ is an upward directed net in $\SG$. Then $[f_\a]$ converges to $[g]$ in $\tau$ if and only if for each $[\vec{r}]=[(r_i)]\in\SS\SQ$, there is $\a_0\in \pzA$ such that if $\a>\a_0$, then $\{i:\|f_\a-g\|_{d_i}<r_i\}\in\FU$.
\end{lemma}
\begin{proof}
    Now $[f_\a]\ra[g]$ in $\tau^\d$ if and only if, for every $\Fr\in\SN_\d$, there is $\a_0\in\pzA$ such that $\a>\a_0$ implies that $[f_\a-g]\in U_\Fr$. But as $\SN_\d=\{\rz r_\om:(r_i)\in\SS\SQ\}$, then $[f_\a]\ra[g]$ if and only if for each $[(r_i)]\in\SS\SQ$, there is $\a_0\in\pzA$ with $[f_\a-g]\in U_{*r_\om}$ for $\a>\a_0$, which by the previous lemma, holds if and only if $[f_\a-g]\in\pzU(r_i)$, and the equivalence follows from the definition of $\pzU(r_i)$.
\end{proof}

    With these preliminaries, we can now state our first equivalence for convergence in $\tau$.

\vspace{10mm}

\begin{theorem}\label{thm: standard serial descrip of tau}
     Suppose that $([f_\Fn]:\Fn\in \pzN)$ is an upward directed net in $\SG$ and $[g]\in\SG$. Then the following are equivalent.
   \begin{enumerate}
      \item[(a)] $[f_\pzn]\ra [g]$ in $\tau$.
      \item[(b)]  Let $[\vec{d}]\in\SS\SQ$ be fixed. Then,  for each  $[(r_i)]\in\SS\SQ$, there is $\Fn_0\in\pzN$ such that $\Fn>\Fn_0$ implies that for some sufficiently large $i_0$, $\|f_\Fn-g\|_{d_i}<r_i$ for $i>i_0$.
   \end{enumerate}
\end{theorem}
\begin{proof}
     We have (b)$\Rightarrow$(a) follows from the previous result. For, let $\d=\rz d_\om$ for some $\om\in\rz\bbn_\infty$, and so let $\FU$ denote the free ultrafilter on $\bbn$ generated by $\om$.  Then for a given $(r_i)\in\SS\SQ$ and $i_0\in\bbn$, the hypothesis in (b) implies that there is $\pzn_0\in\pzN$ such that for $\Fn>\Fn_0$,  $\{i:\|f_\Fn-g\|_{d_i}<r_i\}\supset\{i_0,i_0+1,\ldots\}$, certainly an element of $\FU$. To prove the converse, suppose that (a)$\Rightarrow$(b) does not hold, ie., suppose that (a) holds, but, for some $d_i\ra 0$, there is $[(\ov{r}_i)]\in\SS\SQ$ and $i_0\in\bbn$ such that for every $\Fn\in\pzN$ we have $\|f_\Fn-g\|_{d_i}\geq \ov{r}_i$ for all $i>i_0$. That is, if $\om\in\rz\bbn_\infty$, then for all $\Fn\in\pzN$, $\om\not\in\rz\{i\in\bbn:\|f_\Fn-g\|_{d_i}<\ov{r}_i\}$. On the other hand, as $\d$ is serial, we have that $\d=\rz d_\om$ for some $\om\in\rz\bbn_\infty$ and so $\ov{\Fr}=\rz\ov{r}_\om\in\SN_\d$ by \autoref{lem: 3 items-rel between SQ and SN_d}. But then (a) implies, for $\tau=\tau^\d$, that there is $\Fn_0\in\pzN$ such that, for $\Fn>\Fn_0$ we have $\rz\|\rz f_\Fn-\rz g\|_\d<\ov{\Fr}$. So fixing some $\Fn>\Fn_0$, we have that $\om\in\rz\{i:\|f_\Fn-g\|_{d_i}<\ov{r}_i\}$, a contradiction.
\end{proof}

\subsection{Standard interpretation, $\d$ far from serial}
    In this part, if $L$ denotes Lebesgue measure on the real line, we will consider those $\d\in\mu(0)_+$ such that for measurable $A$, $\d\in\rz A$ only if $A\cap\bbr_+$ satisfies $L(A)>0$. We will examine the standard rendition of $\tau$ in these circumstances.
    For this stronger $\d$, the conditions for net convergence will appear to be much stronger than the previous standard interpretation, but as noted above they are topologically equivalent.

    Let $\mathitbf{Null}$ denote the collection of subsets of $\bbr_+$ of Lebesgue measure $0$,  and let $\mathitbf{Full_0}$ denote the collection of sets of the form $(0,r)\cap(\bbr_+\ssm N)$ for $r\in\bbr_+$ and $N\in Null$. Consider the relation $\La$ on $\bbr_+\x Full_0$ given by $(t,A)\in\La$ if and only if $t\in A$. This is a concurrent relation and so assuming that our nonstandard model of analysis is at least an enlargement, see eg., Hurd and Loeb, \cite{HurdLoeb1985}, we have that there is $\d\in\rz\bbr_+$ such that $\d\in\rz A$ implies $A\in Full_0$. $\d$ is said to be a point of full measure; let $\mathitbf{P_F}\subset\mu(0)_+$ denote these points. Clearly, for $\d\in P_F$, we have  $\d\sim 0$ and $\d\not\in\rz N$ if $N\subset\bbr_+$ has Lebesgue measure $0$. In particular, if $[m]\in\SM$ and $[g]\in\SG$ satisfy $\|\rz g\|_\d<\rz m(\d)$, then as $\d\in \rz A$ for $A\;\dot=\{t\in\bbr_+:\|g\|_t<m(t)\}$, then $A$ has positive Lebesgue measure.
    Note that $\d\in A$ implies that, for $n\in\bbn$, $\d\in \rz(0,1/n]\cap A$ so that $(0,1/n]\cap A$ has positive measure for all $n\in\bbn$.
    Given such a $\d$, let $\boldsymbol{\Xi_\d}=\{A\subset\bbr_+:\d\in\rz A\}$ and note that $\Xi_\d$ is a free ultrafilter.
    Also, as the set of continuous piecewise affine monotone germs, $\SP\SL^0$, give a cofinal class of germs (see \autoref{def: PL^0 germs at 0},  and \autoref{prop: SPSL^0 coinitial in wt(SM)}), eg., ${}_\d\SP$ is a coinitial subset of $\SN_\d$, we can therefore use them in a standard description of $\tau$.
    Hence, analogous to the serial point case, we have the following lemma.
\begin{lemma}\label{lem: tfae for pts of pos meas and PL}
    Let $\d\in P_F$. Suppose that $[m]\in\SP\SL^0$ and $\Fr=\rz m(\d)$. If $[g]\in\SG$, the following are equivalent.
  \begin{enumerate}
     \item [(a)] $[g]\in U^\d_\Fr$,
     \item [(b)]  There is $A\in\Xi_\d$ such that $\|g\|_t<m(t)$ for all $t\in A$.
  \end{enumerate}
\end{lemma}
\begin{proof}
     Given the above remarks, the proof is almost identical to that for \autoref{lem: [f]in U_*r_om <->[f]in pzU(r_i)}.
\end{proof}

   Skipping the intermediary corollary, we can remove the nonstandard dependence (on the choice of $\d$) as follows. In the following, we are assuming  (the directed sets of) our nets are upward directed.

\begin{theorem}\label{thm: standard nonser descrip tau}
    Suppose that $([f_d]:d\in D)$ is a net in $\SG$ and $[g]\in\SG$. Then, the following are equivalent.
   \begin{enumerate}
      \item [(a)] $[f_d]\ra [g]$ in $\tau$.
      \item [(b)] For each $[\pzp]\in\SP\SL^0$, there is $d_0\in D$ and $t_0\in\bbr_+$ such that if $d>d_0$, then $\|f_d-g\|_t<\pzp(t)$ for all $0<t<t_0$.
   \end{enumerate}
\end{theorem}
\begin{proof}
    To prove (a) $\Rightarrow$ (b), suppose to the contrary that (a) holds, but (b) does not. That is, there is $[\pzp_0]\in\SP\SL^0$ and a cofinal subset $D'\subset D$ such that for each $d\in D'$ the following statement $S_d$ is true. $S_d$ is the statement that there are arbitrarily small $t\in\bbr_+$ such that $\|f_d-g\|_t>\pzp_0(t)$. For a given $d\in D'$, the transfer of $S_d$ gives the statement that there are arbitrarily small $\d\in\rz\bbr_+$ with $\|\rz f_d-\rz g\|_\d>\rz\pzp_0(\d)$, ie., for each $d\in D'$, there is $\d\in\mu(0)_+$ with $[f_d-g]\not\in U^\d_{*\pzp_0(\d)}$. But if $\wt{\pzp}=\pzp_0/2$, then for all $\d\in\mu(0)_+$, we have $U^\d_{*\pzp_0(\d)}\supset\SU([\wt{\pzp}])$. That is, for all $d\in D'$, $[f_d-g]\not\in\SU([\wt{\pzp}])$. But, by \autoref{prop: un(SM)^r subset un(SM)(m)} and \autoref{lem: wtSM_r sub wtSM(m)->U_r sub U(m)}, there is $\Fs\in\SN_\d$ such that $\SU([\wt{\pzp}])\supset U_\Fs$. That is, for all $d\in D'$, $[f_d-g]\not\in U^\d_\Fs$, contradicting (a) (for $\tau=\tau^\d$) as $D'$ is cofinal in $D$.
    To prove (b) $\Rightarrow$ (a), note that the transfer of statement (b) for each $[\pzp]\in\SP\SL^0$, implies that for each $[\pzp]\in\SP\SL^0$, there is $d_0\in D$ such that for $d>d_0$, $\|\rz f_d-\rz g\|_\Ft<\rz\pzp(\Ft)$ for all $\Ft\in\mu(0)_+$. But then fixing $\Ft=\d$, for any such $\d\in\mu(0)_+$ gets statement (a) for $\tau=\tau^\d$.
\end{proof}

\subsection{$\tau$ and germ estimates}
    The above standard formulations of convergence in $\tau$ are barely disguised germ expressions. Here, we will summarize the  equivalence given in the previous sections within the framework of germs.
\begin{definition}\label{def: FL: G_0->SM or SSSQ}
    Given $[f]\in{}_m\SG_0$, if $d_1>d_2>\cdots$ is a decreasing sequence in $\bbr_+$ with unique limit $0$, with its function  germ (at $j$ equals infinity) denoted by $[\vec{d}]$, then $j\mapsto\|f\|_{d_j}:\bbn\ra\bbr_+$ has germ (at $j$ equals infinity), $\mathitbf{[(\|f\circ\vec{d}\|_j)]}$. That is, we have a map \boldmath$\boldsymbol{\FL^{[\vec{d}]}}:{}_m\SG_0\ra\SS\SQ$\unboldmath.
    Given \autoref{def: SSSQ, SSSQ_FU, [(r_i)]}, we  have a more refined map \boldmath$\FL^{[\vec{d}]}_\FU:{}_m\SG_0\ra\SS\SQ_\FU$\unboldmath\; given by the composed map $[f]\mapsto[(\|f\circ\vec{d}\|_j)]\mapsto [(\|f\circ\vec{d}\|_j)]_\FU$.
\end{definition}
    With these definitions, the above results can be rendered as follows.
\begin{corollary}\label{cor: all standard descriptions tau}
    Suppose that $([f_\Fn]:\Fn\in \pzN)$ is a net in $\SG$ and $[g]\in\SG$. Then the following are equivalent.
  \begin{enumerate}
   \item [(a)] $[f_\Fn]\ra [g]$ in $\tau$.
   \item [(b)] Given $[\vec{d}]\in\SS\SQ$ and $\FU\in fuf(\bbn)$, for each $[\vec{r}]\in\SS\SQ$, there is $\Fn_0\in\pzN$ such that $\Fn>\Fn_0$ implies that $\FL_\FU^{\vec{d}}([f_\Fn-g])<[\vec{r}]_\FU$.
   \item [(c)] Given $[\vec{d}]\in\SS\SQ$,  for each $[\vec{r}]\in\SS\SQ$, there is $\Fn_0\in \pzN$ such that $\Fn>\Fn_0$ implies that $\FL^{\vec{d}}([f_\Fn-g])<[\vec{r}]$.
   \item [(d)] If $\d\in P_F$ is fixed, then for each $[\pzp]\in\SP\SL^0$, there is $\Fn_0\in\pzN$ such that $\Fn>\Fn_0$ implies that $\FL_{\Xi_\d}([f_d-g])<[\pzp]_{\Xi_\d}$.
   \item [(e)] For each $[\pzp]\in\SP\SL^0$, there is $\Fn_0\in \pzN$ such that $\Fn>\Fn_0$ implies that $\FL([f_d-g])<[\pzp]$
   \item [(f)] For each $[m]\in\wt{\SM}$, there is $\Fn_0\in\pzN$ such that $\Fn>\Fn_0$ implies that $\FL([f_d-g])<[m]$.
  \end{enumerate}
\end{corollary}
\begin{proof}
   The equivalences are a direct consequence of the above corollary, the previous two theorems  and the definitions directly above. \autoref{cor: all standard descriptions tau} implies (a) $\Leftrightarrow$ (b),  \autoref{subsec: stand interpret for d serial} implies (a) $\Leftrightarrow$ (c), \autoref{thm: tau^d on SM indep of d} and \autoref{lem: tfae for pts of pos meas and PL} imply (a) $\Leftrightarrow$ (d)  and \autoref{thm: standard nonser descrip tau}  implies (a) $\Leftrightarrow$ (e). Clearly, the fact that $\SP\SL^0$ is coinitial with $\wt{\SM}$ (\autoref{prop: SPSL^0 coinitial in wt(SM)}) gives the equivalence of (e) and (f).
\end{proof}
   Note that we have arrayed our equivalent descriptions of $[f_d]\ra [f]$ in (b) through (e) from weak to strong. There still seems to be room to find both weaker descriptions than (b) and stronger than (e). Note also that  statement (d) in the equivalence has a familiar framing. Thinking of the partially ordered sets $\SP\SL^0\subset\wt{\SM}$ as the range for our norming function $\FL$, and noting that we have proved that eg., the partially ordered set $\SP\SL^0$ does not have any countable coinitial subsets, the need for (uncountable) nets now seems clear.

\section{Conclusion}\label{sec: conclusion}
\subsection{The topological context}
    Concerning basic \hypertarget{topological}{topological} questions about $\tau=\tau^\d$, we have mentioned the work on germ topologies for very rigid mappings that occurs in eg., Pisanelli, \cite{Pisanelli1976}, and more recently in eg., Glockner, \cite{Glockner2004}. All of these considerations lie in the context of locally convex spaces, too restrictive for our ring topology $\tau$. More particularly, these topologies are typically inductive limit topologies for inductive systems of Frechet spaces arising from natural restriction maps that are continuous inclusions due to the rigid (analytic) nature of the functions, see eg., Meise and Vogt, \cite{MeiseVogt97}, p292 and chapter 29. For us, a restriction map $r_n:F(B_{1/n})\ra\SG$ sends a real valued function $f$ defined on $B_{1/n}$ to its germ $[f]$ at $0$, which is the same as the map $f\mapsto \rz f|B_\d$ for some $\d\in\mu(0)_+$. As noted, for the inductive topology, $r_n$ must be a  continuous injection for all $n$, and the maps $f\mapsto[f]$ can be neither continuous nor injective. Note also that projective limit descriptions of topologies, see pg 279 in \cite{MeiseVogt97}, depend directly on the existence of the fundamental system of seminorms defining the given topology, and so by definition on local convexity. On the other hand, a description in terms of locally convex vector spaces over non-Archimedean fields seems plausible. But, as in Perez-Garcia and Schikhov, \cite{LCSnonArchFlds2010}, the fundamental system of seminorms describing these topologies are real valued and it seems that such a description of this topology can be feasible only if the seminorms are valued in  (the nonnegative values of) a non-Archimedean field with some kind of definable least upper bound property for definable seminorms (see eg., Todorov and Wolf, \cite{oai:arXiv.org:math/0601722}, p3). In this nonstandard vein, it seems to the author that our topology is a kind of transferred (or non-Archimedean) compact-open topology.

    From another angle, we believe that we have recently found in the broader topological framework a natural setting for the germ topology $(\SG,\tau)$. This starts with the simple observation that the map $\mathitbf{\sbbd=\sbbd_\d}:\SG\x\SG\ra\SN_\d$ defined by $\mathitbf{\sbbd_\d([f],[g])=\|\rz f-\rz g\|_\d}$ satisfies all of the properties of a metric space, except that the range is not real and  eg., is  non-Archimedean. Further, the $\sbbd$-balls defining a subbase for this topology are precisely the open sets $U^\d_\Fr$ in our topology $\tau^\d$, ie.,
  \begin{align}
     (\SG,\tau^\d)\;\text{is homeomorphic to}\; (\SG,\sbbd_\d).
  \end{align}

     Topologies defined in terms of metrics with range a partially ordered abelian group (with the appropriate order topology) are said to be (a type of) \emph{generalized metric space}. Kalisch, \cite{Kalisch1946}, shows that all generalized metric spaces for which the range is a ``partially ordered vector group''  are uniform spaces and conversely.  Our topology $\tau$ is homeomorphic to the type of generalized metric space typically called an $\mathitbf{\om_\mu}$\emph{-metric space}, $\Fd:X\x X\ra A$. Here, the range of the metric $\Fd$ is a  totally ordered abelian group $A$ (with the order topology) having the property that  the smallest initial ordinal $\a$ for which there is an $\a$ indexed sequence decreasing to $0$ in $A$, has  cardinality $\om_\mu$ (the $\mu^{th}$ cardinal number). That is, $\aleph_\mu=|\a|$ is the minimal cardinality of a coinitial subset of $A_+=\{a\in A:a>0\}$. Refining the correspondence of Kalisch, Stevenson and Thron, \cite{StevensonThron1969}, prove that these correspond to uniform spaces with linearly ordered bases having minimal cardinality $\aleph_\mu$.  For some characterizations of $\om_\mu$-metric spaces, see also  Artico and Moresco, \cite{ArticoMoresco1981} and most recently Nyikos and Reichel, \cite{NyikosReichel1992}. For us, $A$ is the additive subgroup of $\rz\bbr$ generated by $\SN_\d$, denoted $G(\SN_\d)$, so $A_+=\SN_\d$.

    \emph{Therefore, the work in \autoref{subsec: ord preserv tops not countable} implies that  our spaces $(\SG,\tau^\d)$ are $\om_1$-metric spaces. }

    It appears that this is the first example of a function space that is an $\om_\mu$-metric space (for $\mu>0$), see Hu\v{s}ek and Reichel, \cite{HušekReichel1983}, p 174, as well as the following historical survey in that paper.

    As $\om_\mu$-metric spaces have many of the properties of $\om_0$-metric (ie., metric) spaces (along with some interesting deviations in the nonmetrizable case as summarized in the next paragraph), this contextualization tells us much about $\tau$.
    First, because of the correspondence found by Stevenson and Thron, \cite{StevensonThron1969}, alluded to in the previous paragraph and the work of Souppouris, \cite{GenerzMetrcSpParacpt},
    we find that $\om_\mu$-metric spaces are paracompact, eg., $\mathitbf{(\SG,\tau^\d)}$ \emph{is paracompact}. Furthermore, we have that they are metrizable if and only if $\mu=0$; eg., we have another verification that our topology is not metrizable. Nevertheless, Stevenson and Thron (see p 318 of \cite{StevensonThron1969}) complete a list of theorems mostly proved by Sikorski, \cite{Sikorski1950} that give higher cardinality analogs (for  $\om_\nu$-metric spaces) of a set of theorems satisfied by (metrizable) metric spaces.
    As an example let's just recount the $\om_1$-metric space version of the classic Heine-Borel theorem. First, we need the extended definitions. We say that $\SS\subset\SG$ is $\mathitbf{\om_1}$\emph{-compact} if every $\mathitbf{\om_1}$ \emph{sequence in $\SS$} (ie., an $\SS$ valued sequence indexed by the first uncountable ordinal) has a convergent $\om_1$ subsequence and $\SS$ is said to be $\mathitbf{\om_1}$\emph{-complete} if every $\om_1$ Cauchy sequence in $\SS$ converges.  $\SS$ is said to be $\mathitbf{\om_1}$\emph{-totally bounded} if for each $\Fr\in\SN_\d$, there is a countable subset $\SY\subset\SS$ with $\cup\{B_\Fr(\pzy):\pzy\in\SY\}=\SS$.  Given this,  \emph{Heine-Borel for our $\om_1$-metric space} is the following statement:

    \emph{$\SS\subset (\SG,\tau^\d)$ is $\om_1$-compact if and only if $\SS$ is $\om_1$-complete and $\om_1$-totally bounded. }

    We will give a complete account of these matters and their implications in a latter paper.

    Even more interesting for topological understanding in the case $\mu>0$, eg., for $\tau$, is the fact that, if $\mu>0$, the topology defined by $\Fd:X\x X\ra A$  is uniformly  equivalent to the topology defined by a (generalized) \emph{ultrametric} $\wh{\Fd}:X\x X\ra\wh{A}$, eg., $\wh{\Fd}(x,y)\leq\max\{\wh{\Fd}(x,z),\wh{\Fd}(y,z)\}$ (the strong triangle inequality; see eg., \cite{ArticoMoresco1981}, p 243 and de Groot, \cite{deGroot1956}). We say generalized since it's not hard to see that the work in \autoref{subsec: ord preserv tops not countable} implies that  $\wh{A}$ will be non-Archimedean and, as noted in the previous paragraph, the typical treatment of ultrametrics, as in eg., Perez-Garcia and Schikhov, \cite{LCSnonArchFlds2010}, has them real (eg., Archimedean) valued (but see the paper of de Groot just noted). Emphasizing, our situation begins with a (generalized) metric that does not, a priori, satisfy the strong triangle inequality; but, the uncountable coinitial nature of our $A=G(\SN_\d)$ implies the existence of a non-Archimedean valuation $v:A\ra\wh{A}$. (A general construction of the valuation map $v: A\ra\wh{A}$ is given, eg., in Artico and Moresco, \cite{ArticoMoresco1981}, p 242-243. In a later paper we will give a concrete construction useful in our context.)
    With the non-Archimedean valuation we can get our ultrametric by defining $\wh{\sbbd}=v\circ\sbbd$.

    \emph{In particular, we have that $(\SG,\sbbd)$ is homeomorphic to a so called nonarchimedean topological space $(\SG,\wh{\sbbd})$, see eg., \cite{NyikosReichel1992}.}

    Some of the principal properties of topologies defined by ultrametrics are independent of the minimal cardinality of coinitial subsets of the range group, $\wh{A}$, of the ultrametric (again see de Groot, \cite{deGroot1956} and compare with Artico and Moresco, \cite{ArticoMoresco1981}, on the one hand and Perez-Garcia and Schikhov, \cite{LCSnonArchFlds2010}, on the other.). For simplicity, here we will summarize only basic properties of this type.
    The central facts are the following: if  $\wh{B}_1,\wh{B}_2$ are balls in $(\SG,\wh{\sbbd})$ and $\wh{B}_1\cap\wh{B}_2\not=\emptyset$, then either $\wh{B}_1\subset \wh{B}_2$ or $\wh{B}_2\subset\wh{B}_1$. Directly related is the fact that all balls are open and closed in the $\wh{\sbbd}$ topology. From this one can verify the following.

    \emph{$(\SG,\wh{\sbbd})$ is a $0$-dimensional topological space, ie., there is a base for the topology consisting of open and closed sets. In particular, our topology is totally disconnected. }

    Further investigation of these topologies will occur in a later paper.
\subsection{Prospects and speculation}
   The author began this work with a pair of long term goals in mind. First, stimulated by his work on regularizing families of locally Euclidean topological groups, he wondered if it was possible to develop a topological notion of nearness for germs of local metric groups. Second, the author has wondered for some time about a nonstandard approach to  developing a workable framework for the study of the rings and modules of differentiable or smooth map germs. As alluded to in the introduction, there is some evidence that such a development might be useful in this arena, if not in a broader arena. The author believes that this paper, in demonstrating the possibility of a good topological theory of function germs in the simplest setting, is evidence that both are possible.
   Furthermore, we have some preliminary ideas on how to extend the framework here toward these goals, but all such is still quite tentative.

   Certain things seem clear, but are left to do. First,  the subspace topology induced on the subring  of (real) $C^k$ or smooth germs will not be first countable. Clearly with differentiable germs, we can also define the *supremum norms on various *Taylor polynomials, and although it seems clear that analogs of much done here will go through, we are uncertain about a good approach to this.  On the other hand, in the case of analytic germs, we should get a first countable topology; ie., we will need only countable analogs of the $\SN_\d$'s.     It's clear that they will neither be closed in $\SG^0$, nor a dense subspace, although, in the case of complex analytic germs, it seems reasonable to expect the ring to be complete (seemingly a consequence of, eg., a multivariable *Cauchy integral formula on $\d$ balls or polydisks).
   Considering the possibility of topologizing other types of germs, we can speculate that using the transfer of Hausdorff distance (as a distance between the transfers of closed sets intersected with a fixed infinitesimal $\d$ ball), and the right choice for an analog of $\SN_\d$,  one can in get a nontrivial Hausdorff topology on germs at $0$ of closed subsets of $\bbr^n$. (The latter part of Stevenson and Thron, \cite{StevensonThron1969}, develops the Hausdorff metric on the set of closed subsets of an $\om_\mu$-metric space.)  It does not seem unreasonable to assert that analogs of \autoref{cor: SG_0 -> F(B_delts) is R alg isomorph} and  \autoref{prop: ptwise bndd subset unif bndd} will imply that this topology is independent of choice of $\d$.
   Within the context of considerations on germs of closed subsets, it now seems clear that a more natural approach to these topologies from the perspective of convergence would have been through the device of Cauchy nets. It's clear that one could restate the main convergence results of this paper within this context. But, beyond a strict continuity with future work, not much is gained from these wholesale changes. Nonetheless, Cauchy nets will be central to future works in topologizing germs.

   Finally, the  concluding section gives standard descriptions of $\tau$, in contrast to the body of the text that depends fundamentally on nonstandard constructions. It is known that, vaguely speaking, a statement in mathematics with a nonstandard proof indeed has a standard proof. What is not generally known, eg., see the paper of Henson and Keisler \cite{HensonKeisler1986}, is that the standard proof will, in principle, use constructions from the next `higher level' of mathematics. For example, in our context we looked at the values of sets of germs at a given, fixed, nonstandard point. In the standard world, points in $\mu(0)_+$ become equivalence classes of sets of points and sets of germs become sets of equivalence classes of functions. Without saying more, it should be apparent that although a standard proof exists, in principle; it seems that a straightforward conversion of our proof will be orders of magnitude more difficult to construct and certainly less intuitive.



\section{Appendix: A brief impression of the nonstandard framework}\label{appendix: nsa impression}
    Good introductions to (Robinson-Luxemburg-Keisler) nonstandard mathematics  abound, see eg.,  Lindstrom's, \cite{Lindstrom1988}, and Henson's, \cite{Henson1997} and also the argument in Farkas and Szabo, \cite{FarkasSzabo} and, of course, the encyclopedic standard is the (difficult) text of Stroyan and Luxemburg, \cite{StrLux76}. Given the extensive references, we opt instead for a tour for the uninitiated of some of the structures in the nonstandard environment, drifting toward making sense of the omnipresent $\SN_\d$ of the present text. As a possible primer, one might read the following to get some sense of the pertinent structures and then go to, eg., Lindstrom's article to see some careful statements and proofs.

    In the small, the enhancement gained with  nonstandard mathematics can be seen in a nontrivial  ultrapower of $\bbr$, here denoted in the nonstandard fashion by $\rz\bbr$. We will flesh this in by viewing $\rz\bbr$ within the context of its collection of internal, standard and external subsets.
    Given any fixed free ultrafilter $\FU$ defined on $\bbn$,  this {\it ultrapower} is defined as the set of equivalence classes, $\bk{r_i}$, of all sequences $r:\bbn\ra\bbr$ (denoted $F(\bbn,\bbr)$). Here $(r_i)$ is equivalent to $(s_i)$, $(r_i)\sim (s_i)$, if and only if $\{i:r_i=s_i\}\in\FU$.  Although the ring, order theoretic, etc., properties of the range, $\bbr$, only minimally endows $F(\bbn,\bbr)$  with the structure of a partially ordered ring, one can prove that the properties of $\FU$ actually imply that this quotient, $\rz\bbr$, naturally has `all' of the properties of the range $\bbr$, eg.,  $\rz\bbr$ is a totally ordered real closed field. Furthermore, this new field, $\rz\bbr$ is a much enhanced copy of $\bbr$. In fact, the map $a\mapsto\bk{a}$, sending $a$ in $\bbr$ to the equivalence of the constant sequence ($a_i=a$ all $i$), gives an ordered field injection $\bbr\ra\rz\bbr$ and the image of $\bbr$ (denoted ${}^\s\bbr$) is `bounded' and `sparse' in its image, eg., $\rz\bbr$ has infinite and infinitesimal elements. By definition, a positive element $\bk{r_i}$ is an {\it infinitesimal} ($\mu(0)$ will denote these) if for each positive $t\in\bbr$, $\bk{r_i}<\bk{t}$ which by definition holds if $\{i:r_i<t\}\in\FU$. In fact, we get infinitesimals of numerous relative non-Archimedean magnitudes by, eg.,  choosing  $r_i\ra 0$ at various rates and noting that $\FU$ contains complements of finite subsets of $\bbn$.

    Next, we will analogously consider the {\it transfer} of the set $\SP(\bbr)$ of subsets of $\bbr$. That is, suppose  $I\SP(\rz\bbr)$ denotes the set of the equivalence classes of maps $S:\bbn\ra\SP(\bbr)$ with a similar equivalence relation recipe: $(S_i)\sim (T_i)$ if $\{i:S_i=T_i\}\in\SU$. Then, in analogy with $\rz\bbr$ having the properties of $\bbr$ (the range space in the ultrapower), the properties of $\FU$ imply that these equivalence classes have the `same' properties as $\SP(\bbr)$.    In particular, as bounded elements of $\SP(\bbr)$ have suprema, it's also true that the bounded (now in terms of bounds in $\rz\bbr$\;!) elements of $I\SP(\rz\bbr)$  have  nonstandard suprema, written $\rz\sup\bk{A_i}$. In the text, we write our nonstandard element in some non-Roman script, eg., denoting $\bk{A_i}$ by, say, $\SA$.

     Continuing our perusal of the properties of the subsets of $\rz\bbr$ by looking at the range space properties of an ultrapower, we will now look at the possible {\it transfer} of the natural set theoretic relations between our two range spaces, $\bbr$ and $\SP(\bbr)$. This will allow us to discern all subsets of $\rz\bbr$ as being either internal, in $I\SP(\bbr)$, or external.  Just as the elements of $\SP(\bbr)$ can be seen as the subsets of $\bbr$, it indeed holds  that $I\SP(\rz\bbr)$ can be naturally viewed so that its elements are (precisely the {\it internal}) subsets  of $\rz\bbr$. Using the properties of $\FU$, one can check that defining $\bk{b_i}\in\bk{B_i}$ if $\{i:a_i\in B_i\}\in\FU$, gives a well defined and consistent notion of internal set membership.
    More concretely, we have the special subclass of these internal subsets given by the equivalence classes of the maps $S$ that are constant (denoted by $\rz A$ if $S_i=A$ for all $i$), which are the {\it standard subsets} of $\rz\bbr$, denoted ${}^\s\SP(\bbr)$.   For example, the equivalence class of the constant sequence $A_i=(-\pi,2/3)$ for all $i$, ie., $\rz(-\pi,2/3)$, consists of all equivalence classes $\bk{a_i}$ where $\{i:-\pi<a_i< 2/3\}\in\FU$, eg., it contains all translates of infinitesimals by $\bk{a}$ for $a\in(-\pi,2/3)$.
    Clearly, the standard set $\rz(-\pi,2/3)$ is much `thicker' than the embedded copy of $(-\pi,2/3)$, ie.,  ${}^\s(-\pi,2/3)$. So, although, as with $\bbr\ra{}^\s\bbr$, we   have the bijection $\SP(\bbr)\ra {}^\s\SP(\bbr)$, in this situation the image of elements of $\SP(\bbr)$ are greatly extended. As a further example, let us specify  numerous general internals intervals that, eg.,. are all `infinitesimally close to' our standard interval. If $a_i,b_i\in\bbr$ for $i\in\bbn$ with $a_i$ strictly decreasing to $-\pi$ and $b_i$ strictly increasing to $2/3$, then all such intervals $\SB=\bk{(a_i,b_i)}\in I\SP(\rz\bbr)$ contain the same standard points as $\rz(-\pi,2/3)$. Note that the {\it standard part} of such an $\bk{(a_i,b_i)}$ includes these points (back in $\bbr$) as well as those standard points infinitesimally close, and therefore will also include $-\pi$ and $2/3$.   As noted with the elements of $\rz\bbr$ above, but even more so with  $I\SP(\rz\bbr)$, we have a wild array of extremal elements that, nonetheless, carry all of the `well stated' properties of $\SP(\bbr)$. A mild, but useful,  example of such might be an internal $\pzA\subset\rz\bbr$ that is  finite as a nonstandard set, called {\it *finite}, for which we have ${}^\s\bbr\subset\pzA$. (To make these one needs a `bigger' $\FU$ in our ultrapower construction.)  Yet, in spite of the above mentioned wildness, being internal, these  sets have `all' of the properties of our sets, a special case of {\it transfer} (and this is independent of the choice of $\FU$). In particular, $\SA$ has all of the `well stated properties of a finite set of real numbers'. On the other hand, there are  {\it external subsets} of $\rz\bbr$ for which {\it transfer} does not work, eg., the set of infinitesimals, $\mu(0)$, cannot be written as $\bk{A_i}$ and neither can ${}^\s\bbr$, the isomorphic copy of $\bbr$ in $\rz\bbr$. For these, transfer does not hold, eg., although $\mu(0)$ and ${}^\s\bbr$ are bounded subsets of $\rz\bbr$, and bounded internal sets have *suprema by transfer, we will shortly indicate that neither of these have suprema.

    If $F(n,1)$ denotes the set of functions $g:\bbr^n\ra\bbr$, then one further extension of the above recipe is to consider $\FU$ equivalence classes of maps $f:\bbn\ra F(n,1)$ where again $f=(f_i)$ is defined to be equivalent to $g=(g_i)$ if $\{i:f_i=g_i\}\in\FU$, where as above, we will let $\Ff=\bk{f_i}$ denote the equivalence class containing $(f_i)$. Pedantically, these too have `all'  of the properties of the range space $F(n,1)$. In particular, they have natural representations as (internal!) functions from $\rz\bbr^n$ to $\rz\bbr$. (It's not too hard to see that the obvious definition is the proper and well defined one: $\bk{f_i}(\bk{a_i})\dot=\bk{f_i(a_i)}$.)
    Let's give another simple  example of transfer.
    If $C^0(n,1)\subset F(n,1)$ is the subset of continuous functions, then the set $\rz C^0(n,1)\subset\rz F(n,1)$, consisting of those $\Ff=\bk{f_i}$ satisfying $\{i:f_i\in C^0(n,1)\}\in\FU$, has all of the `well stated' properties that $C^0(n,1)$ has. For example, these {\it *continuous functions}\; satisfy the nonstandard version of the intermediate value theorem. If $\Ff\in\rz C^0(n,1)$ and $\a,\b\in\rz\bbr$ are such that $\Ff(\a)<0<\Ff(\b)$, then there is $\g$ with $\a<\g<\b$ satisfying $\Ff(\g)=\bk{0}$. In sharp contrast to the fact that $\rz C^0(n,1)$ has all of the nice `well stated'  properties of $C^0(n,1)$ and in analogy with the wild properties of elements of $\rz\bbn,\rz\bbr$ and $I\SP(\rz\bbr)$, there are, for example, elements  $\Ff\in\rz C^0(n,1)$ such that $\rz|\;\Ff|$ takes arbitrarily large values on each standard interval in $\rz\bbr$ but, nevertheless satisfies $\Ff|{}^\s\bbr\equiv 0$! (To get such $\Ff$, one needs larger more robust $\FU$, giving a type of {\it saturated ultrapower}.)

    Before we move to  our discussion of the sets $\SN_\d$, we need a view of more general standard functions specific to our example. First, note that for each $f\in F(n,1)$, we have the corresponding equivalence class, $\rz f$, of the constant sequence $f_i=f$, for all $i$. That is, from  the given function $f$ defined on $\bbr^n$, we get the corresponding standard function $\rz f$ defined on $\rz\bbr^n$ (with values in $\rz\bbr$).
    (Note that, as with the injections onto the isomorphic external subset $\SP(\bbr)\ra{}^\s\SP(\bbr)\subset I\SP(\rz\bbr)$, we have  $F(n,1)\ra{}^\s F(n,1)\subset\rz F(n,1)$.)
    Next, if  $\bbx=F(\SP(\bbr)_{bdd},\bbr)$ denotes the set of functions from  $\SP(\bbr)_{bdd}$ (bounded subsets of $\bbr$) to $\bbr$, then, again following our recipe, we can consider the set, $\rz\bbx$, of $\FU$ equivalence classes of maps $T:\bbn\ra\bbx$. If $\ST=\bk{T_i}\in\rz\bbx$, then as our range space $\bbx$ is given by maps from $\SP(\bbr)_{bdd}$ to $\bbr$, then, following the process done with elements of $\rz F(n,1)$, we have that $\ST$ can be naturally viewed as a mapping from $I\SP(\rz\bbr)_{*bdd}$ to $\rz\bbr$. Note this  automatically gives  internal domains and ranges for the map $\ST$,  and eg., gives the existence of *suprema for *bounded internal subsets of $\rz\bbr$. That is, $\sup$ is an element of $\bbx$ and so $\rz\sup=\bk{\sup}$ is one of the standard elements of $\rz\bbx$, eg., is an internal map from $I\SP(\rz\bbr)_{*bdd}$ to $\rz\bbr$ assigning *suprema to these internal sets. Furthermore, $\rz\sup$ satisfies $\rz\sup \SA\cup\SB=\max\{\rz\sup\SA,\rz\sup\SB\}$ if $\SA,\SB$ are internal and *bounded (ie., in $I\SP(\rz\bbr)_{*bdd}$). But if $\SA_j$, $j\in\bbn$, are elements of $I\SP(\rz\bbr)_{*bdd}$ uniformly bounded in $\rz\bbr$,  then $\rz\sup(\cup_j\SA_j)$ may not even be defined as $\cup_j\SA_j$ may not be internal. For example, $\SA_j=\rz\{x:|x|<j\}$ is even standard, but $\cup_j\SA_j$ is the set of {\it nearstandard} points in $\rz\bbr$, denoted $\rz\bbr_{nes}$,  and is therefore bounded, but has no *supremum.  On the other hand, if $\bbx= F(\bbn,\SP(\bbr)_{bdd})$ and $\SE\in\rz\bbx$, ie., an {\it internal} sequence of elements of $I\SP(\rz\bbr)$, then, the transfer of the usual property holds. That is, for  $\Fj\in\rz\bbn$, letting $\SE_\Fj=\SE(\Fj)$ and $\Fs_\Fj=\rz\sup\SE_\Fj$ for each $\Fj\in\rz\bbn$, we have $\rz\sup(\cup_\Fj\SE_\Fj)$ exists and equals $\rz\sup\{\Fs_\Fj:\Fj\in\rz\bbn\}$.
    Finally, suppose that we have a map $S:\bbr_+\ra\SP(\bbr)$. Then the previous discussion implies that the map $\rz S:\rz\bbr_+\ra I\SP(\rz\bbr)$ is internal, eg., its image is contained in $I\SP(\rz\bbr)$ and so its values are internal subsets of $\rz\bbr$.



     Having completed the above cursory introduction, fix a positive infinitesimal $\d\in\mu(0)$ and we now want to say something about the sets $\SN_\d$. Let $F_0(n,1)$ be those $f\in F(n,1)$ satisfying for all sufficiently small $t>0$, $t\mapsto\|f_0\|_t\in\SM$, see \autoref{def:  SM and SM^0}. First, we wish to see how, for a given positive $\d\in\mu(0)$ and $f_0\in F_0(n,1)$, $\rz\|\rz f_0\|_\d$ is a well defined element of $\rz[0,\infty)$. To this end, let $S_{f_0}:\bbr_+\ra\SP(\bbr)$ be $S_{f_0}(t)=\{|f_0(x)|:x\in B_r\}$, and note  from the previous paragraph, $\rz S_{f_0}(\d)\subset\rz\bbr$ is internal. But, with a little work (and details of transfer not given above), we get $\rz S_{f_0}(\d)=\rz\{|\rz f_0(\xi)|:\xi\in B_\d\}$. So the previous paragraph implies that $\rz\|\rz f_0\|_\d$, defined here to be $\rz\sup$ of this set, is a well defined element of $\rz[0,\infty)$. (In the body of the paper, we often leave out some *'s.) A much easier approach would be to transfer the map $\sup\circ S_{f_0}:\bbr_+\ra[0,\infty)$ and evaluate at $\d$, but we believe this `components' approach is more instructive. Now, $\SN_\d$ is defined to be the set of these numbers as $f_0$ varies over germs, which is clearly the same set if we let $f_0$ vary in $F_0(n,1)$. $\SN_\d$ is an external subset of $\rz[0,\infty)$ basically for the same reasons that ${}^\s F_0(n,1)$ is an external subset of $\rz F_0(n,1)$ (and ${}^\s\bbr$ is external in $\rz\bbr$). An easy  way to verify that $\SN_\d$  is external is via the following simple criterion (stated for subsets of the internal set $\rz\bbr$ but clearly adaptable for the others). If $\SH\subset\rz\bbr$ is nonempty, *bounded above with the property that for some for some integer, $n$, greater that one, $\rz n\a\in\SH$ whenever $\a\in\SH$; then $\SH$ is external. Setting, $\SH$ equal ${}^\s\bbr$ or $\SN_\d$ as subsets of $\rz\bbr$, or ${}^\s F_0(n,1)$ as a subset of $\rz F(n,1)$, its clear both conditions are satisfied. Note that, given the remarks on internality and $\rz\sup$ above, one could construct a simple proof of the criterion for externality.

      Although outside the impressionistic gist of this summary, we believe that it's important to mention two techniques that are seldom used in nonstandard arguments. The first concerns the standard implications of a relation that needs an enlargement: $\d\ggg\k$ for positive $\d,\k\in\mu(0)$. Here, this device was central to a finding a useful property that a germ must satisfy on an arbitrary fixed infinitesimal ball about $0$ in $\rz\bbr^n$ that is equivalent to its continuity on some ball $B_r$ for some $r\in\bbr_+$, ie., is a continuous germ. See \autoref{prop: germ continuity from k<<<d} and the previous lemma.
    The second technique focuses on the implications of the simple fact that if $A\subset\bbr$ has $\rz A\cap(\rz\bbr\ssm{}^\s\bbr)\not=\emptyset$, then $A$ must be infinite. This becomes quite useful if, as in \autoref{prop: ptwise bndd subset unif bndd}, one defines $A$ in a bootstrap fashion, eg., here the elements of $P$ are functions whose transfers satisfy a bound at some nonstandard value. We use this device in our dissertation for other purposes, see \cite{McGaffeyPhD}, and believe it can have significant utility.
\section{Commonly used terms and definitions}
    We include here an index of definitions and terminology used often along with the page where these are defined.
    \begin{multicols}{2}
\begin{enumerate}
    \item [pg\pageref{def: set of (f,U) s.t. U=domf})] $\rz\bbr_+,\rz\bbn_\infty, B_r=B^n_r, \rz B_{\d}=B_\d$
    \item [pg\pageref{def: set of (f,U) s.t. U=domf})] $\mu(0),\mu(0)_+,0<\d\sim0$
    \item [pg\pageref{def: set of (f,U) s.t. U=domf})] $\SG=\SG_{n,1},\SG^0, [f]$
    \item [pg\pageref{def: [f]<[g]})] $[f]<[g]$
    \item [pg\pageref{def: e is [f]-good})] $\k\lll\d,\om\lll\Om$,
    \item [pg\pageref{def: e is [f]-good})]  has strong good $\d$ numbers
    \item [pg\pageref{def: e is [f]-good})] $\k$ strongly $[f]$-good for $\d$
    \item [pg\pageref{def: SN_delta families of numbers})] $\|f\|_r,\rz\|f\|_\d=\rz\|\rz f\|_\d$
    \item [pg\pageref{def: SN_delta families of numbers})] ${}_m\wt{\SG}_0,\;{}_m\SG_0,\;{}_m\wt{\SG}_{[f]}$
    \item [pg\pageref{def: SN_delta families of numbers})] $\wh{\SN}_\d,\wt{\SN}_\d,\SN_\d,\SN_\d^0$
    \item [pg\pageref{def:  SM and SM^0})] $\wt{\SM},\SM,\SM^0$
    \item [pg\pageref{def:{}_dSM, {}_dwtSM, etc})] ${}_\d\wt{\SM},{}_\d\SM,{}_\d\SM^0,{}_\d\SB$
    \item [pg\pageref{def: good coinitial subsets})]  has good coinitial subsets for $\d$
    \item [pg\pageref{def: U_Fr})] $U^\d_\Fr=U_\Fr$
    \item [pg\pageref{def: of topology  at 0 germ})] $\tau^\d_0,\tau^\d,\tau_0,\tau$
    \item [pg\pageref{def: SU([m]), SM([m]), M^Fr, etc})] $\SU([m])=\SU([m])_u$, $\SU([m])_\ell$
    \item [pg\pageref{def: SU([m]), SM([m]), M^Fr, etc})] $\SB^\Fr,\SB_\Fr$ for $\SB=\wt{\SM},\SM,\SM^0$
    \item [pg\pageref{def: SU([m]), SM([m]), M^Fr, etc})] $\SB([m])_u,$\! for\! $\SB\negmedspace=\negmedspace\wt{\SM},\SM,\SM^0$
    \item [pg\pageref{def: SU([m]), SM([m]), M^Fr, etc})] $\SB([m])_{\!\ell}$\! for\! $\SB\negmedspace=\negmedspace\wt{\SM},\SM,\SM^0$
    \item [pg\pageref{def: FL:SG->SM})] $\FL([g]), \FL([\SB])$
    \item [pg\pageref{def: of coinitial partial order relations})] $J$ is coinitial in $P$
    \item [pg\pageref{def: of coinitial partial order relations})] $J$ is coinitial with $K$ in $P$
    \item [pg\pageref{def: of coinitial partial order relations})] $J$ and $K$ are coinitial in $P$
    \item [pg\pageref{def: of coinitial partial order relations})] $\SV$ convergently coinitial in $S$
    \item [pg\pageref{def: PL^0 germs at 0})] $\SP\SL^0$, ${}_\d\SP$
    \item [pg\pageref{def: asymp tot ord set of germs})] asymptotically totally ordered
    \item [pg\pageref{def: pzT preserves germ scales})] preserves $(\SG,\La)$ germ scales
    \item [pg\pageref{def: SG{n,p}, SG^0{n,p}})] $\SG_{n,p},\SG^0_{n,p}$
    \item [pg\pageref{def: rc_[h],lc_[h]})] $rc_{[h]},lc_{[h]}$
    \item [pg\pageref{def: SSSQ, SSSQ_FU, [(r_i)]})] serial point, $\FU=fil(\om)$
    \item [pg\pageref{def: SSSQ, SSSQ_FU, [(r_i)]})] $\SS\SQ$, $[\vec{r}]=[(r_i)]$
    \item [pg\pageref{def: SSSQ, SSSQ_FU, [(r_i)]})] $\SS\SQ_\FU,[\vec{r}]_\Fu$, $\stackrel{\FU}{<}$
    \item [pg\pageref{def: FL: G_0->SM or SSSQ})] $\FL^{\vec{d}},\FL^{\vec{d}}_\FU$
\end{enumerate}
\end{multicols}

\end{document}